\documentclass[11pt,leqno]{article}
\RequirePackage[leqno,cmex10]{amsmath}
\RequirePackage[arxiv,noinfoline]{imsart}
\RequirePackage[OT1]{fontenc}
\RequirePackage{amsthm,amssymb}
\usepackage{mathrsfs,enumerate,graphicx,dsfont,hyperref}
\usepackage[initials]{amsrefs}
\usepackage[letterpaper,margin=2.8cm]{geometry}

\usepackage{enumerate}
\newenvironment{enumeratei}{\begin{enumerate}[\upshape (i)]\setlength{\itemsep}{4pt}}{\end{enumerate}}

\numberwithin{equation}{section}
\numberwithin{figure}{section}
\numberwithin{table}{section}
\theoremstyle{plain}
\newtheorem{thm}{Theorem}[section]
\newtheorem{lem}[thm]{Lemma}

\newtheorem{prop}[thm]{Proposition}
\newtheorem{defn}[thm]{Definition}

\theoremstyle{definition}

\theoremstyle{remark}
\newtheorem{rem}[thm]{Remark}
\renewcommand{\le}{\leqslant}
\renewcommand{\ge}{\geqslant}

\newcommand{\ra}{\rangle}
\newcommand{\la}{\langle}

\newcommand{\eps}{\varepsilon}

\newcommand{\set}[1]{\left\{#1\right\}}

\newcommand{\ie}{\emph{i.e.,}\ }

\newcommand{\iid}{\emph{i.i.d.}~}

\def\qed{\hfill $\blacksquare$}
\let\ga=\alpha \let\gb=\beta \let\gc=\gamma \let\gd=\delta 
\let\gf=\varphi      
    \let\gs=\sigma  \let\gth=\vartheta
  
 \let\gD=\Delta  \let\gL=\Lambda

\newcommand{\cG}{\mathcal{G}}

\newcommand{\vI}{\mathbf{I}}

\newcommand{\vO}{\mathbf{O}}

\newcommand{\vY}{\mathbf{Y}}

\newcommand{\vp}{\mathbf{p}}

\newcommand{\vv}{\mathbf{v}}
\newcommand{\vy}{\mathbf{y}}

\newcommand{\dN}{\mathds{N}}

\newcommand{\dR}{\mathds{R}}

\newcommand{\sA}{\mathscr{A}}

\newcommand{\sE}{\mathscr{E}}

\DeclareMathOperator{\pr}{\mathds{P}}

\def\eqd{\,{\buildrel d \over =}\,}

\def\tI{{\tilde I}}
\def\tr{{\tilde r}}
\newcommand{\sF}{\mathscr{F}}
\newcommand{\sG}{\mathscr{G}}
\def\beq{ \begin{equation} }
 \def\eeq{ \end{equation} }
 \def\beqx{ \begin{equation*} }
 \def\eeqx{ \end{equation*} }
 \def\beqa{\begin{eqnarray}}
 \def\eeqa{\end{eqnarray}}
 \def\beqax{\begin{eqnarray*}}
 \def\eeqax{\end{eqnarray*}}
\def\eqd{\,{\buildrel d \over =}\,}
\newcommand{\IO}[1]{#1^{\text{in,out}}}
\newcommand{\RBN}{\text{RBN}}
\newcommand{\ola}{\overleftarrow}
\newcommand{\ol}{\overline}
\newcommand{\I}[1]{#1^{\text{in}}}
\newcommand{\OU}[1]{#1^{\text{out}}}

\newcommand{\bE}{\mathbb{E}}

\newcommand{\mvpi}{\boldsymbol{\pi}}
\newcommand{\mvmu}{\boldsymbol{\mu}}
\newcommand{\mvgd}{\boldsymbol{\gd}}
\newcommand{\sref}[1]{\text{\ref{#1}}}
\newcommand{\sI}{\mathscr{I}}

\newcommand{\tY}{\tilde{Y}}
\newcommand{\ep}{\epsilon}

\begin{document}

\begin{frontmatter}
\title{The order-chaos phase transition for a general class of complex Boolean networks}
\runtitle{Complex Boolean networks}

\begin{aug}
\author{\fnms{Shirshendu} \snm{Chatterjee}\thanksref{m1}\ead[label=e1]{shirshendu@cims.nyu.edu}}
\runauthor{S. Chatterjee} \affiliation{New York
University\thanksmark{m1}}

\address{Courant Institute of Mathematical Sciences \\
New York University \\
251 Mercer Street \\
New York, NY 10012, USA\\
\printead{e1}}
\end{aug}

\date{\today}
\begin{abstract}
 We consider a model for heterogeneous {\it gene regulatory networks} that
 is a generalization of the model proposed by Chatterjee and Durrett \cite{CD11}
 as an ``annealed approximation" of Kauffmann's \cite{K69} random Boolean networks. In
 this model, genes are represented by the nodes of a random directed
 graph on $n$ vertices with specified in-degree distribution $\I\vp$ (resp.~out-degree
 distribution $\OU\vp$ or joint distribution $\IO\vp$ of in-degree and out-degree), and the expression bias (the expected fraction of 1's in the Boolean functions) $p$ is same for all nodes.
 Following \cite{CD11} and a standard practice in the physics literature, we use
 a discrete-time {\it threshold contact process} with parameter $q=2p(1-p)$ (in which a vertex with at least one {\it occupied} input at
 time $t$ will be occupied at time $t+1$ with probability $q$, and
 {\it vacant} otherwise) on the above random graph to approximate the
 dynamics of the Boolean network.
 We show that there is a parameter $r$, which has an explicit expression in terms of certain moments of $\I\vp$ (resp.~$\OU\vp$ or $\IO\vp$), such that, with probability tending to 1 as $n$ goes to infinity, if $r\cdot 2p(1-p)>1$,
 then starting from all occupied sites the threshold contact process maintains a positive ({\it quasi-stationary}) density $\rho(\I\vp)$ (resp.~$\rho(\OU\vp)$ or $\rho(\IO\vp)$)
 of occupied sites for time which is exponential in $n$, whereas if $r \cdot 2p(1-p)<1$, then the persistence time of the threshold contact process is at most logarithmic in $n$.
 These two phases correspond to the {\it chaotic} and {\it ordered} behavior of the gene networks.
\end{abstract}



\end{frontmatter}

 \section{Introduction}
 Experimental evidence \cite{AO03,shmulevich2002binary,tabus2006normalized} suggests that in various
 biological systems, the complex kinetics of genetic control is
 reasonably well approximated by {\it Boolean network} models. These models were first formulated by Kauffman \cite{K69}, and over the last few years they
 have received significant attention, both at the level of model formulations and numerical simulations (see e.g. the surveys
 \cite{KCA02, POGL09,K93,huang1999gene,somogyi1996modeling} and the references therein) and at the rigorous level (see e.g. \cite{CD11, FK88}).
 The basic model can be described as follows. Genes are represented by the nodes of a directed network $\sG_n=\sG_n([n],\sE_n)$
 on $n$ vertices, where $[n]:=\set{1,2,\ldots, n}$ denotes the vertex set and $\sE_n$ denotes the edge set for $\sG_n$.
 The state $\eta_t(x)$ of a node $x\in [n]$ at time $t=0, 1, 2, \ldots$ is either
 1 (`on') or 0 (`off'), and each node $x$ receives input from the nodes which point {\bf to}
 it in $\sG_n$, namely
 \[ Y^x := \{y \in [n]: \la y,x\ra \in \sE_n\}. \]
 $Y^x$ is called the {\it input set} and its members are called the {\it input nodes} for
 $x$. The states $\{\eta_t(x)\}_{t \ge 0, x \in [n]}$ evolve according to the update rule
 \beq \label{rbn1}
 \eta_{t+1}(x) = f_x( (\eta_t(y), y\in Y^x) ), \qquad x\in [n],
 \eeq
 where each $f_x:\set{0,1}^{|Y^x|}\to \set{0,1}$ is
 some time-independent Boolean function defined on the set of states of the input nodes for $x$.
 Here and later we use $|A|$ to denote the size of a set $A$.

In order to understand general properties of such dynamical systems,
various {\it random Boolean network} models have been formulated,
which form an important subfamily of these models.
 The simplest such model with parameters $r$ (number of inputs per node)
 and $p$ (expression bias), which we denote by $\RBN^1_{r,p}$, consists of the following specification of the constructs in the above general model.
 The base network $\sG_n$
 is constructed by choosing, for each node $x\in [n]$,
 the input set $Y^x=\{Y_1(x), \ldots, Y_r(x)\}$ consisting of $r$ distinct input nodes uniformly from $[n]\setminus\{x\}$.
 The values $f_x(\vv), x \in [n], \vv\in \{ 0, 1\}^{|Y^x|},$ are assigned independently and each equals 1 with probability $p$.
 The dynamics then proceeds as in \eqref{rbn1} from a specified starting configuration $\set{\eta_0(x): x\in [n]}$ at time $t=0$.
 Note that the functions $f_x$ and the graph $\sG_n$ are fixed at time $0$ and then the dynamics of this system is {\bf deterministic}.

   Kauffman introduced $\RBN^1_{r,1/2}$ in \cite{K69}, and that model has been analyzed in detail for $r=1$ \cite{FK88}. The general model $\RBN^1_{r,p}$
 has been studied extensively via simulations (see e.g.
 \cite{KCA02,pomerance2009effect}) and using heuristics from Statistical Physics (see e.g. \cite{drossel20083,drossel2005number}). It has been argued in \cite{DP86} that
  the behavior of $\RBN^1_{r,p}$ undergo a phase transition for $r\ge 3$,
 and
 \beq \label{ph tran1}
 \text{ the {\it order-chaos} phase transition curve for } \RBN^1_{r,p} \text{ is given by } 2p(1-p)\cdot r = 1 \eeq
 in the sense that  $\RBN^1_{r,p}$ is ``ordered'' (the configuration of zeros and ones rapidly converges to a fixed point or attractor) for $2p(1-p)\cdot r<1$, whereas $\RBN^1_{r,p}$ is
 ``chaotic'' (the configuration keeps on changing for an exponentially long time) when $2p(1-p)\cdot r>1$.
 This phase transition picture has recently been proved for an
 ``annealed approximation" of the deterministic dynamical system
 \cite{CD11,MV11}, which we shall describe soon.

 Although $\RBN^1_{r,p}$ has been studied extensively, the model deals with an idealized setting of homogeneous networks, where every vertex has the same in-degree. This assumption constrains the model in the context of biological applications where such networks have quite heterogeneous degrees \cite{RA05}.
 A natural question is whether similar phase transitions occur in the
 context of
 heterogeneous complex networks, and if so, how does the corresponding phase transition curves depend on the underlying parameters describing the networks.
 The update rule in case of heterogeneous networks is similar to that for $\RBN^1_{r,p}$, namely for each node $x\in
 [n]$, given the input set $Y^x$, the values
 $f_x(\vv)$ for $\vv\in \set{0,1}^{|Y_x|}$ and $x \in [n]$ are chosen via independent coin flips
 with success probability $p$.

 Heterogeneous network models as the underlying base graph have been considered in the physics and biology literature. Two classes of such network models have been formulated in the literature.
\begin{enumeratei}
    \item {\bf Networks with prescribed in-degree:} A number of authors (e.g. \cite{SL94, LS97, FH01,
     AC03}) have considered directed network models with prescribed in-degree distribution. Here one starts with a probability mass function $\I\vp:= \set{\I p_k}_{k\geq 1}$ and constructs a directed random network where for $k\geq 0$, the proportion of vertices which have input set of size $k$ is approximately $\I p_k$ in the large network limit $n\to\infty$. The precise method of construction is described in Section \ref{model}.  
     We denote this model with in-degree distribution $\I\vp$ and expression bias $p$ by
     $\RBN^2(\I\vp,p)$.
     It has been argued in the above papers that
     \beqa
     && \text{the order-chaos phase transition curve for $\RBN^2(\I\vp,p)$ is } 2p(1-p)\cdot \I r = 1, \notag \\
     && \text{ where }  \I r := \sum_k k \I p_k \text{ is the average in-degree.}\label{ph tran2} \eeqa

 \begin{rem} An analogous model can be built, where we now specify the out-degree distribution $\OU\vp=\{\OU p_k\}_{k\ge 1}$.  We denote such a model with out-degree distribution
 $\OU \vp$ and expression bias $p$ by $\RBN^3(\OU p, p)$. The
 analogous conjecture in this regime is
 \beqa
 \text{the order-chaos phase transition curve for $\RBN^3(\OU\vp,p)$ should be } 2p(1-p)\cdot \OU r = 1, \notag \\
 \text{ where }  \OU r := \sum_k k \OU p_k \text{ is the average out-degree.}\label{ph tran3}
 \eeqa
 \end{rem}

 \item {\bf Networks with prescribed joint distribution of in-degree and out-degree:} The most general and complex form of this model is where one incorporates correlation between the in-degree and out-degree via prescribing their joint distribution \cite{LR08}. Here one starts with a bivariate probability mass function $\IO \vp=\{\IO p_{k,l}\}_{k,l\ge 1}$ and constructs a random graph where the asymptotic density of vertices in $\sG_n$ with in-degree $k$ and out-degree $l$ is $\IO p_{k,l}$ as $n\to\infty$. We defer a complete description of the construction to Section \ref{model}. We denote this model by $\RBN^4(\IO\vp,p)$.  It is argued non-rigorously in \cite{LR08} that for this model
 \beqa
 && \text{the order-chaos phase transition curve for $\RBN^4(\IO\vp,p)$ is } 2p(1-p) \cdot \frac{\IO r}{\I r}=1, \notag \\
 && \text{ where } \IO r := \sum_{k,l} kl\IO p_{k,l}, \I r :=
 \sum_{k,l} k\IO p_{k,l}. \label{ph tran4}\eeqa
 \end{enumeratei}

\subsection{Annealed approximations to Boolean networks}
\label{sec:anneal} Proving rigorous results about the formulated
discrete dynamical systems turns out to be quite hard.
 In order to understand the conjectured phase transitions in
 \eqref{ph tran2}, \eqref{ph tran3} and \eqref{ph tran4} rigorously, we consider a different process called the \emph{threshold contact process}. To motivate this process and the connection to the boolean network model $\set{\eta_t(x):x\in \sG_n, t\geq 0}$, first consider the process
 $\{\zeta_t(x), x\in [n], t\ge 1\}$, where $\zeta_t(x)=1$ if $\eta_t(x) \neq \eta_{t-1}(x)$
 and $\zeta_t(x)=0$ otherwise.  Fix node $x\in [n]$. Suppose at least one of the inputs $y\in Y^x$
 changes its state between time epochs $t-1$ and $t$ so that $\eta_t(y) \ne \eta_{t-1}(y)$.
 Then the state of node $x$ at time $t+1$ is
 computed by looking at a different entry of $f_x$.  Ignoring
 the fact that we may have used this entry before, one approximately has
\[\pr(\zeta_{t+1}(x) = 1| \zeta_t(y)=1 \text{ for at least one } y\in Y^x ) =  2p(1-p)\]
If $\zeta_t(y) = 0$ for all $y\in Y^x$, then obviously
$\zeta_{t+1}(x) = 0$. This dynamics motivates the following process
which will be the main aim of this study.
\begin{defn} \label{def:threshold}
    The threshold contact process on a finite directed graph, where $Y^x$ denotes the set of input vertices to the node $x$, is the discrete time Markov process $\set{\xi_t(x):x\in [n]}_{t\geq 0}$ with evolution dynamics
\beqa
 \pr\left(\left. \xi_{t+1}(x)=1 \right|
 \xi_t(y)=1 \text{ for at least one } y\in Y^x \right)& = & 2p(1-p),
 \text{ and} \notag \\
 \pr\left(\left. \xi_{t+1}(x)=0 \right|
 \xi_t(y) = 0 \text{ for all } y\in Y^x \right)& = & 1,
 \label{threshcp} \eeqa
 Conditional on the state at time $t$, the decisions on the values of
 $\xi_{t+1}(x), x \in [n]$, are independent.
 \end{defn}
  This process has been called the {\it annealed approximation}
 to the random Boolean network model \cite{DP86}. For the rest of the study we will write $q=2p(1-p)$.
 For the threshold contact process, we shall prove that the conjectures
  in \eqref{ph tran2}, \eqref{ph tran3} and \eqref{ph tran4}
  do represents the order-chaos phase transition for this process.

 \subsection{Construction of random directed networks} \label{model}
 In this section, we provide precise mathematical formulation of the
 underlying network models for $\RBN^2$ and $\RBN^4$ ($\RBN^3$ can be considered
 as a special case $\RBN^4$).

\subsubsection{Construction of the network for $\RBN^2$}

 For $\RBN^2$, let $\I\vp=\{\I p_k\}_{k\ge 1}$ be a prescribed in-degree distribution with
 $\I r := \sum_k k \I p_k   < \infty$.
 Following \cite{CD11},  we construct the random directed graph $\sG_n = \sG_n([n], \sE_n)$ with in-degree distribution $\I\vp$ as follows.
 First we choose the in-degrees $ I_1, I_2, \ldots, I_n $ independently with common distribution $\I\vp$. For each node
 $x\in[n]$ we choose the corresponding input set $Y^x:=\{Y_1(x), Y_2(x), \ldots, Y_{I_x}(x)\}$ by choosing
 $|I_x|$ distinct nodes uniformly from $[n]\setminus\{x\}$.
 Finally we place oriented edges from these chosen vertices {\bf to} $x$ to obtain the edge set $\sE_n$ of the graph $\sG_n$, $\sE_n=\{\la Y_i(x),x\ra: x\in [n], 1\le i\le I_x\}$.

 It is easy to see that given the in-degree sequence $\{I_i\}$, the number of choices for the input sets $\set{Y^x: x\in [n]}$  is $\prod_{i=1}^n {n-1 \choose I_i}$. Writing $e_{zx}$ for the indicator denoting the presence or absence of a directed
 edge from node $z$ to node $x$, $\pr_{2,\vI}$ for the conditional (``quenched ") distribution of $\sG_n$ given the in-degrees $\set{I_x:x\in [n]}$ and $\pr_{2,n}$ for the unconditional (``annealed'') distribution of $\sG_n$,
 \beq \label{P2def}
 \pr_{2,n}(\cdot) := \sum_{\vI\in\dN^n} \pr_{2,\vI}(\cdot) \;
 \I\vp_{\otimes n}(\vI),\eeq
 where $\I\vp_{\otimes n}$ is the product
 measure on $\dN^n$ with marginal $\I\vp$, and
 \begin{equation}
 \label{bP_2def} \pr_{2,\vI} (e_{zx}, 1\le z,
     x\le n)=1/
     \prod_{i=1}^n {n-1 \choose I_i} \text{ if $e_{z,x}\in\{0,1\},
     e_{x,x}=0, \sum_{z=1}^n e_{zx}=I_x \forall~ x\in [n]$, }
\end{equation}
and $\pr_{2,\vI}(e_{zx}, 1\le x, z\le n)=0$ otherwise.

%

 \subsubsection{Construction of the network for $\RBN^4$} \label{cons:RBN4}
 We follow the procedure of Newman, Strogatz and Watts \cite{NSW01,
 NSW02}. Given a prescribed joint distribution for in-degree and out-degree
 $\IO\vp=\{\IO p_{k,l}\}_{k,l\ge 1}$ with
 $\IO r := \sum_{k,l} kl \IO p_{k,l}   < \infty$ and $\I r := \sum_{k,l} k \IO p_{k,l} = \sum_{k,l} l \IO p_{k,l}=:\OU r$, let $\set{(I_i, O_i)}_{1\leq i\leq n}$ be \iid with common distribution $\IO\vp$; here $I_x$ and $O_x$ denote the
 in-degree and out-degree  of node $x$ respectively.
 We need to condition on the event
 \beq \label{Endef} E_n := \left\{ \sum_{i=1}^n I_i = \sum_{i=1}^n O_i\right\}
 \eeq
 to have a valid degree sequence. Having chosen
 the degree sequence $\{(I_x,O_x)\}_{x\in [n]}$, allocate $I_x$ many ``inward arrows"
 and $O_x$ many ``outward arrows"  for node $x$.  Pick a
 uniform random matching between the set of inward arrows and outward arrows. If one of the inward arrows of $x$ is matched with one of
 the outward arrows of $z$, then we let $\la z,x\ra \in \sE_n$. Let
 $\pr_{4,\vI,\vO}$ denote the conditional (``quenched") distribution of $\sG_n$
 given $\{(I_i,O_i)\}$. We also condition on the event
 \beq \label{Fndef}
 F_n:=\{ \sG_n \text{ is simple}\},\eeq
 \ie it neither contains any
 self-loop at some vertex, nor contains multiple edges between two
 vertices. So if $\pr_{4,n}$ denotes the unconditional (``annealed") distribution of
 $\sG_n$, then
 \beq
 \pr_{4,n}(\cdot) = \sum_{\vI,\vO\in\dN^n}
 \pr_{4,\vI,\vO}(\cdot| F_n) \IO\vp_{\otimes n}((\vI,\vO) | E_n),
 \label{P4def}\eeq
 where $\IO\vp_{\otimes n}$ is the product measure on
 $(\dN^2)^n$ with marginal $\IO\vp$.

 \begin{rem}
 For $\RBN^3$, given the prescribed out-degree distribution $\OU\vp=\{\OU
 p_k\}_{k\ge 1}$ with $\OU r := \sum_k k \OU p_k   < \infty$, one can follow the construction of $\sG_n$ for $\RBN^4$ (see Section \ref{cons:RBN4}) corresponding to any $\IO\vp$ with marginal $\OU\vp$.
 \end{rem}


 \subsection{Dynamics}
Once the base network has been fixed via one of the above
constructions, we shall be interested in properties of the threshold
contact process $\set{\xi_t}_{t\geq 0}$ as in Definition
\ref{def:threshold} with paramter $q=2p(1-p) \in (0, 1/2)$. We shall
often view this as a set valued process $\xi_t:=\set{x\in [n]:
\xi_t(x)=1 }$. We shall sometimes refer to this as the set of
occupied sites at time $t$. We will write $\pr_{\sG_n,q}$ for the
distribution of the threshold contact process
 $\{\xi_t\}_{t\ge 0}$ with parameter $q$, conditioned on the base graph $\sG_n$.
 For a fixed set $A\subseteq [n]$, we shall write $\set{\xi_t^A}_{t\geq 0}$ for
 the process started with $\xi_0= A$.


 \subsection{Main results} \label{sec:results}

 For a probability distribution $\mvmu$ on $\{0, 1, \ldots\}$ and $q\in[0,1]$ let
 $\mvpi(\mvmu,q) \in [0,1]$ denote the survival probability for the
 branching process with offspring distribution $(1-q)
 \mvgd_0+q\mvmu$ starting from one individual. From the branching
 process theory,
 \beq \label{pidef}
 \mvpi(\mvmu,q)=1-\theta, \text{ where $\theta \in [0,1]$ is the minimum value satisfying } \theta=1-q+\sum_{k\ge 0}q\mvmu\{k\} \theta^k.\eeq
 Using the above ingredients we now present our main result.

 \begin{thm}\label{thm:RBN2}
 For any probability distribution $\I\vp=\{\I p_k\}_{k \geq 0}$ on $\dN$, let $\pr_{2,n}$ be the probability distribution
 (as defined in \eqref{P2def}) on the set of random directed graphs on $n$ vertices having in-degree distribution $\I\vp$.
 Suppose $\I\vp$ has mean $\I r$ and finite second moment, $\I p_0 = \I
 p_1=0$, and $q \in (0,1)$ satisfies
 $q\in (1/\I r,1/2)$. Let $\pi:=\mvpi(\I\vp,q)$ be the branching process survival probability as defined in
 \eqref{pidef}. Then $\pi > 0$, and for any $\eps>0$ there is a constant $\gD(\eps)>0$ and a `good' set of graphs $\cG_n$ satisfying $\pr_{2,n}(\cG_n) =1-o(1)$ such that if $\sG_n \in \cG_n$, then
 \[ P_{\sG_n,q}\left(\inf_{t \le \exp(\gD n)}\frac{|\xi^{[n]}_t|}{n} \ge \pi-\eps\right) \to 1 \text{ as } n\to\infty.\]
 Moreover, if $q$ satisfies $q\I r < 1$, then there is a constant
 $C(q, \I r)>0$ such that $P_{\sG_n,q}(\xi^{[n]}_{C\log n} \ne
 \emptyset)=1-o(1)$ for all $\sG_n \in \cG_n$.
 \end{thm}

 The above theorem proves \eqref{ph tran2} for the threshold-contact
 process. Next we move to $\RBN^3$.

%

 \begin{thm}\label{thm:RBN4}
 For any probability distribution $\IO\vp=\{\IO p_{k,l}\}_{k, l \geq 0}$ on $\dN^2$, let $\pr_{4,n}$ be the probability distribution
 (as defined in \eqref{P2def}) on the set of random directed graphs on $n$ vertices for which the joint distribution of in-degree and
 out-degree is $\IO\vp$.
 Suppose, both the marginal distributions corresponding to $\IO\vp$ have equal mean $\I r$ and finite second moment, $\IO p_{k,l} = 0$ whenever $k\leq 1$, and $\IO r:=\sum_{k,l} kl\IO p_{k, l}$. Let $q \in (0,1)$ be such that
 $\I r/\IO r < q < 1/2$ and $\pi:=\mvpi(\tilde\vp,q)$ be the branching process survival probability as defined in
 \eqref{pidef}, where $\tilde p_k:= (\I r)^{-1}\sum_{l\geq 0} l \IO p_{k,l}$. Then $\pi>0$ and for any $\eps>0$ there is a constant $\gD(\eps)>0$ and a `good' set of graphs $\cG_n$ satisfying $\pr_{4,n}(\cG_n) =1-o(1)$ such that if $\sG_n \in \cG_n$, then
 \[ P_{\sG_n,q}\left(\inf_{t \le \exp(\gD n)}\frac{|\xi^{[n]}_t|}{n} \ge \pi-\eps\right) \to 1 \text{ as } n\to\infty.\]
 Moreover, if $q$ satisfies $q\I r < 1$, then there is a constant
 $C(q, \I r)>0$ such that $P_{\sG_n,q}(\xi^{[n]}_{C\log n} \ne
 \emptyset)=1-o(1)$ for all $\sG_n \in \cG_n$.
 \end{thm}

 The above theorem proves \eqref{ph tran4} for the threshold-contact
 process.

 \subsection{Discussion} \label{sec:disc}
 It is needless to say that the main challenge to prove persistence in the supercritical regime
 for these models is the heterogeneity of the networks that we
 consider. So, the techniques used in \cite{CD11}, which are based
 on ``isoperimetric inequalities" for random regular graphs, and in
 \cite{MV11}, which are applicable only for random regular graphs,
 can not be used. Since the number of subsets of size $O(n)$ is
 super-exponentially large and not much is understood about the
 roles of different subsets of occupied sites in the prognosis of
 the dynamics beyond the level of the ``first moment method", the
 `coupling with truncated branching process' technique used in the article seems to be
 the only effective way in the supercritical regime.

 However, this technique does not work that well in the critical
 regime where the parameters $p, \I r$ and $\IO r$ satisfy the equality in \eqref{ph tran2} and
 \eqref{ph tran4}. Based on the behavior of critical contact process
 on $d$-dimensional torus, one expects polynomial persistence of
 activity in this case.

 \subsection{Organization of the paper}
 The remainder of the paper is organized as follows.
 In Section \ref{prelim}, we describe the dual for the threshold
 contact process, which will play a crucial role in our argument,
 and give quantitative estimates for approximating the local
 neighborhood in the dual graph of certain small subsets of
 vertices. Then in Section \ref{ingredients} we mention some more
 ingredient lemmas which are used later. Section \ref{good graph}
 contains description of the `good' graph that appears in the
 theorems and proof of the fact that it has probability $1-o(1)$.
 Fin ally in Section \ref{proof} we put all the ingredients together
 to have the proof of the main theorems.

 \section{Preliminaries} \label{prelim}
 Before jumping into the core of the proof we need some preliminary
 facts. We begin this section with the definition of the dual process
 for the threshold contact process which will play a major role in
 proving the main results. We also collect asymptotic properties of
 local neighborhoods of the random graph models
 $\set{\RBN^i:i=2,3,4}$.

\subsection{Dual coalescing branching process}
For a given directed graph $\sG_n=([n], \sE_n)$, let $\ola
\sG_n=([n], \ola\sE_n)$ be the directed graph obtained by reversing
the edges, \ie $\ola\sE_n:=\{\la x,y\ra: \la y,x\ra \in \sE_n\}$.
Write $x\to z$ for a directed edge $(x,z) \in \ola\sE_n$. We shall
occasionally refer to $z$ as a child of $x$ in $\ola \sG_n$. Now for
the threshold contact process $\set{\xi_t: t\geq 0}$ on the original
graph $\sG_n$, the dual process $\{\ola\xi_t: t\ge 0\}$, is the
coalescing branching process on the graph $\ola \sG_n$ whose
dynamics we now describe. The process starts from some specified set
of occupied vertices $\ola \xi_0 = B$. For each $t\geq 1$, each site
of $\ola\xi_t$ gives birth independently with probability $q$ at
time $t$. If $x \in \ola\xi_t$ gives birth, all of its children are
included in $\ola\xi_{t+1}$.  More precisely, every vertex $x\in
\ola \xi_t$ gives birth with probability $q$, independent across
vertices. For every $z\in [n]$,  $\ola \xi_{t+1}(z) = 1$ if there
exists $x\in \ola \xi_t$ which gives birth at time $t$, else
$\ola\xi_{t+1}(z) = 0$.  Writing $\set{\ola \xi_t^B: t\geq 0}$ for
the coalescing branching process started with $\ola\xi_0 = B$, it is
easy to check \cite{G79} the following duality relation. For any
$t\geq 0$ and sets $A, B\subseteq [n]$ we have
\begin{equation}
\label{eqn:eqn-duality}
    \pr\left(\xi^A_t \cap B \ne \emptyset \right) = \pr\left( \ola\xi^B_t \cap A \ne \emptyset \right).
 \end{equation}
 This will be our core technical tool in proving the main results
 about the original process $\xi_t$.

We will need the following notation in the sequel. For the graph
$\ola \sG_n$ and $U\subset [n]$ define
 \[ U^{*1} := \{z \in [n]: x \to z \text{ for some } x \in U \}.\]

 \subsection{Local neighborhoods} \label{loc nei}
 Next, we need to understand the structure of the neighborhood of a
 small set of vertices of $\ola \sG_n$. The goal is to see whether the oriented neighborhood of a typical small vertex set
 contain an oriented forest whose offspring distribution is close to
 the out-degree distribution in $\ola\sG_n$.

 For $A \subset [n]$ we let $\ola Z^A_0=A$ and for $l\ge 1$ let
 \[ \ola Z^A_l:=\{z \in [n]: \text{ there is an oriented path in $\ola\sG_n$ of length $l$ from some $x \in A$ to $z$}
 \}. \]
 Let $\{K_m\}_{m=1}^n$ be a sequence of numbers which will be specified later (see \eqref{K_mdef}).
 Using $\sqcup$ to denote disjoint union, we introduce the following coupling between the directed subgraph of $\ola\sG_n$ induced by $\cup_{i= 0}^{K_{|A|}} \ola Z^A_i$
 and a tree
 $\{Z^A_t, 0\le t\le K_{|A|}\}$ along with partitions $Z^A_t=C^A_t\sqcup O^A_t \sqcup R^A_t$, where $C^A_t$, $O^A_t$ and $R^A_t$ represent
 `closed', `open' and `removed' sites at level $t$ respectively.
 Let $A=\{u_1, \ldots, u_{|A|}\}$.
 For the root level of the tree we choose $Z^A_0=A$ and $C^A_0=R^A_0=\emptyset$. The sites in $Z^A_0$ are labeled $u_1, \ldots, u_{|A|}$.
 For each $t\ge 0$ every site of $O^A_t$ mimics the corresponding vertex in $\ola Z^A_t$ with same label, and so a site of $O^A_t$  having label
 $u$ gives birth to $I_u$ many children at level $t+1$. The new born sites at level $t+1$ are
 assigned the same labels following those of $\ola Z^A_{t+1}$.
 Writing
 \[ \Pi(l,A) \text{ for the subset of $A$ consisting of $l\wedge |A|$ elements with minimum
 indices},\]
 we scan the sites of $Z^A_{t+1}$ in an increasing order of labels, and
 define
 \[R^A_{t+1} := Z^A_{t+1}\setminus\Pi(2r |O^A_t|,Z^A_{t+1}).\]
 For a site in $Z^A_{t+1} \setminus R^A_{t+1}$, we say that a ``collision" has occurred if its label either matches with that of a site in $\cup_{s=0}^t O^A_s$, or has already been found while scanning the sites of level $t+1$.
 We include all of these
 sites in $C^A_{t+1}$. If collision does not occur at a site, we include that in $O^A_{t+1}$.

 For $u\in Z^A_t$, let $\ola u^t\in A$ denote the label of the unique ancestor of $u$ having level 0.
 For any subset $B\subset A$ and $t\ge 1$ let $Z^{A,B}_0=O^{A,B}_0=B$ and
 \begin{align*}
 Z^{A,B}_t & := \{u\in Z^A_t: \ola u^t \in B\} \\
 C^{A,B}_t & := \Pi\left(2r|O^{A,B}_{t-1}|, Z^{A,B}_t\right) \cap C^A_t,
 O^{A,B}_t:=\Pi\left(2r|O^{A,B}_{t-1}|, Z^{A,B}_t\right) \cap O^A_t.
 \end{align*}

 Each site in $O^A_t$ corresponds to a unique vertex in $\ola Z^A_t$ with the same label. Note that this map from $O^A_t$ to $\ola Z^A_t$ may not be onto because of collisions and removal of sites.

 The law of $\ola\sG_n$ induces the law of $\{Z^A_t\}$ along with its partitions.
 We identify these two laws. Now our aim is to estimate the probability of collision, and then understand the offspring distribution
 in the above forest. We write
 \begin{align}
 & \ol I_n:= \frac 1n \sum_{z=1}^n I_z, \quad \ol{I_n^2}:=\frac 1n \sum_{z=1}^n
 I_z^2, \quad
 \ol O_n:= \frac 1n \sum_{z=1}^n O_z, \quad \ol{O_n^2}:=\frac 1n \sum_{z=1}^n
 O_z^2 \quad \ol{IO_n}:= \frac 1n \sum_{z=1}^n I_zO_z, \\
 & \gth(m) := m + m\sum_{l=1}^{K_m} (2r)^l.  \label{gthdef}
 \end{align}
 For $\eta \in (0,1)$ and any probability distribution
 $\mvmu$, define
 \beq \label{invdist}
 \Gamma(\eta,\mvmu) := \int_0^\eta \mvmu^\leftarrow(1-t)\; dt,
 \text{ where } \mvmu^\leftarrow(t) := \inf\{y \in \dR:
 \mvmu((-\infty, y]) \ge t\} \text{ for } t \in [0,1].\eeq
 Recall from  Section \ref{model} that $I_x$ denotes the in-degree of $x$ and $\{Y^x_1, \ldots, Y^x_{I_x}\}$
 denotes the set of input nodes for $x$ in $\sG_n$.

 \begin{lem}[For $\RBN^2$] \label{RBN2}
 If $\I\vp$ has finite second moment,
 then  there is a constant $C_{\sref{RBN2}}>0$ and a set of in-degree sequences $\sI_n \subset \dN^n$ such that
 $\I\vp_{\otimes n}(\sI_n)=1-o(1)$ and  $\vI \in \sI_n$ implies
 \beqax
 (1) && \pr_{2,\vI}( x \in C^A_t) \le 2\gth(|A|)/n \text{ for any } A \subset [n], x \in Z^A_t \setminus R^A_t \text{ and } 1\le t\le K_{|A|}       \\
 (2) && \pr_{2,\vI}\left[(Y^x_i, x \in [n], 1\le i\le I_x) \in
 \cdot\right] \le C_{\sref{RBN2}} \pr_{2,\vI}\left[(\tY^x_i, x \in [n], 1\le i\le I_x) \in
 \cdot\right], \eeqax
 where $\{\tY^x_i\}_{x\in[n], i\le I_x}$ are \iid with common distribution
 $Uniform([n])$.
  \end{lem}

 \begin{proof} We take
 \[ \sI_n := \left\{\max_z I_z \le n^{3/4}, \ol I_n <c1, \ol{I_n^2} < c_2\right\}, \]
 where $c_i=2\sum_k k^i\I p_k$. Obviously $\I\vp_{\otimes n}(\sI_n)=1-o(1)$.

 \noindent
 (1). Note that $\sum_{t=0}^{K_|A|}|Z^A_t\setminus R^A_t| \le \gth(|A|)$. So if $\vI \in \sI_n$, then it is easy
 to see from the construction of $\sG_n$ under the law $\pr_{2,\vI}$ that for any $1\le t\le K_{|A|}$ and $x \in Z^A_t \setminus R^A_t$,
 \[ \pr_{2,\vI}(x \in C^A_t) \le \frac{\gth(|A|)}{n-\max_z I_z} \le 2\gth(|A|)/n. \]

 \noindent
 (2). It is easy to see that
 \[ \left(Y^x_i, 1\le i\le I_x, x \in [n]\right) \eqd
  \left(\left.\tY^x_i, 1\le i\le I_x, x \in [n]\right| \tY^z_i \ne \tY^z_j \ne z \forall z \in [n], i \ne
 j\right),\]
 so it suffices to show that $\vI \in \sI_n$ implies $\pr(\tY^z_i \ne \tY^z_j \ne z \forall z \in [n], i \ne j) \ge c$ for some constant.
 Using the inequality $1-x\ge e^{-2x}$ for small $x>0$, we see that
 if $\vI \in \sI_n$, then
 \beqx \pr_{2,\vI}(\tY^z_i \ne \tY^z_j \ne z \forall
 z \in [n], i \ne j) =\prod_{z=1}^n \prod_{i=1}^{I_z} (1-i/n) \ge
 \exp\left(-2\sum_{z=1}^n \sum_{i=1}^{I_z} (i/n)\right)=\exp(- \ol I_n-\ol{I_n^2}) \ge e^{-c_1-c_2}
 \eeqx
 for large enough $n$.
 \end{proof}

 \begin{lem}[For $\RBN^4$] \label{RBN4}
 \begin{enumerate}
 \item[] If the marginal distributions $\I\vp$ and $\OU\vp$ have
 finite second moment, then there is a constant $C_{\sref{RBN4}}>0$ and a set of degree sequences
 $\sA \subset (\dN^2)^n$ such that
 $\IO\vp_{\otimes n}(\sA| E_n)=1-o(1)$ and for all $(\vI,\vO) \in \sA$,
 \beqax
 (1) && \pr_{4,\vI,\vO}(x \in C^A_t) \le 2\Gamma(2\eps,\OU\vp)
 \text{ for any } A \subset [n], x \in Z^A_t \setminus R^A_t \text{ and } 1\le t\le K_{|A|} \\
 (2) && \pr_{4,\vI,\vO}\left[\left.(Y^x_i, 1\le i \le I_x, x \in [n]) \in
 \cdot\right| F_n \right] \\
 && \le  C_{\sref{RBN4}} \pr_{4,\vI,\vO}\left[\left.\left(\tY^x_i, 1\le i \le I_x, x \in [n]\right)\in \cdot\right|
 \sum_{x=1}^n\sum_{i=1}^{I_x} \mathbf 1_{\{\tY^x_i = z\}} =O_z \forall z \in [n]
 \right],
 \eeqax
 where $\{\tY^x_i\}$
 are \iid with common distribution $\sum_{z\in[n]} O_z \mvgd_z/\sum_{x\in[n]} O_x$.
 Moreover, if $\OU p_k \sim ck^{-\ga}$ for some $\ga>2$, then $\Gamma(\eta,\OU\vp)$
 has the same behavior as in Lemma \ref{RBN3}.
 \item[] If $N=rn$ for some constant $r$ and $\{\tY_i\}_{i=1}^N$ are \iid with common distribution
 $Multinomial(1; \ga_1, \ga_2, \ldots, \ga_n)$, then for any small $\eps > 0$,
 \beqax
 (3) && P\left(\left(\tY_1, \ldots, \tY_{\eps N}\right) \in \cdot\left|
 \sum_{i=1}^N \mathbf 1_{\{\tY_i = z\}} =N\ga_z \forall z \in [n]\right.\right)
 \le (1+o(1)) \exp(-\eps\log(1-\eps) N) P\left(\left(\tY_1, \ldots, \tY_{\eps n}\right) \in
 \cdot \right) \text{ for some constant $C$}, \\
 (4) && \left|\left|P\left(\left(\tY_1, \ldots, \tY_{\eps n}\right) \in \cdot\left|
 \sum_{i=1}^N \mathbf 1_{\{\tY_i = z\}} =N\ga_z \forall z \in [n]\right.\right)
 -  P\left(\left(\tY_1, \ldots, \tY_{\eps n}\right) \in
 \cdot \right)\right|\right|_{TV} \le O(\eps^2 n)+o(1).
 \eeqax
 \end{enumerate}
 \end{lem}

 \begin{proof}
 For $\gth(\cdot)$ as in \eqref{gthdef} and $\Gamma(\cdot,\cdot)$ as in \eqref{invdist} we take
 \[ \sA_n := E_n \cap \left\{\frac{\sum_{i=1}^{\gth(|A|)} O_{n,n-i+1} }{\sum_z O_z} \le 4\Gamma(2\eps, \OU\vp)/\OU r \right\}, \]
 where $O_{n,1} \le O_{n,2} \le \cdots \le O_{n,n}$ are the order statistics for $O_1, \ldots, O_n$.
 In order to prove $\IO\vp_{\otimes n}(\sA_n| E_n)=1-o(1)$,
 we apply Theorem 1 of \cite{S74} for the function
 \[J(t) = \begin{cases}
          0 & \text{ for } 0\le t\le 1-2\eps \\
          1 & \text{ for } 1-\eps \le t\le 1 \\
          (t-1)/\eps+2 \text{ for } 1-2\eps \le t\le 1-\eps
 \end{cases}.
 \]
 and the \iid random variables $O_1, \ldots, O_n$.
 Since $\OU\vp$ has finite second moment, it can be checked easily that
 the quantity given in (10) of \cite{S74}, $\gs^2(J,\OU\vp)$ is
 finite. This together with Theorem 4 of \cite{S74} implies
 \[ \IO E_{\otimes n}(S_n-\mu)^2 =O(1/n), \text{ where }\
  S_n := \frac 1n \sum_{i=1}^{n} J(i/(n+1)) O_{n,i} \text{ and } \mu := \int_0^1 J(t) (\OU\vp)^\leftarrow(t) \; dt
 \]
 is as in (11) of \cite{S74}.
 Note that
 \[ \int_{1-\eps}^1 (\OU\vp)^\leftarrow(t)\; dt \le \mu \le \int_{1-2\eps}^1 (\OU\vp)^\leftarrow(t)\; dt, \text{ which means } \Gamma(\eps,\OU\vp) \le \mu \le \Gamma(2\eps,\OU\vp).\]
 Consequently, using Chebyshev inequality
 \[ \IO\vp_{\otimes n}(S_n > 2\mu) =O(1/n),\]
 which in turn implies
 \[ \IO\vp_{\otimes n}\left(\frac 1n \sum_{i=1}^{\gth(|A|)} O_{n,n-i+1} > 2\Gamma(2\eps,\OU\vp)\right) =O(1/n),\]
 and hence
 \[ \IO\vp_{\otimes n}\left(\left.\frac 1n \sum_{i=1}^{\gth(|A|)} O_{n,n-i+1} > 2\Gamma(2\eps,\OU\vp)\right| E_n\right) \le O(\sqrt n) \IO\vp_{\otimes n}\left(\sum_{i=1}^{\gth(|A|)} O_{n,n-i+1} > 2\Gamma(2\eps,\OU\vp)\right) =
 o(1).\]
 The last inequality follows from the fact that $\IO\vp_{\otimes n}(E_n)=O(1/\sqrt n)$ by the local central limit theorem.

 On the other hand,
 using Chebyshev inequality,
 \[ \IO\vp_{\otimes n}\left(\ol O_n \le \OU r/2 \right) = O(1/n), \text{ and so } \IO\vp_{\otimes n}\left(\left.\ol O_n \le \OU r/2 \right| E_n \right) = o(1).\]
 Combining the last display with \eqref{orderstattail} we see that $\IO\vp_{\otimes n}(\sA)=1-o(1)$.

 \noindent
 (1). It is easy to see from the construction of $\sG_n$ that each site of $Z^A_t$ has label $z$ with probability
 $\le O_z/\sum_i O_i$. So, if we write the labels of the sites in
 $\cup_{t=1}^{K_{|A|}} Z^A_t \setminus R^A_t$ in an increasing order, then a collision can occur at the $k$-th site (in this
 ordering) with probability $\le
 \sum_{i=1}^{|A|+k-1} O_{n,n-i+1}/\sum_z O_z$. Therefore, for any $1\le t\le
 K_{|A|}$ and $x\in Z^A_t \setminus R^A_t$
 \[ \pr_{4,\vI,\vO}(x \in C^A_t) \le \frac{\sum_{i=1}^{\gth(|A|)} O_{n,n-i+1}}{\sum_z
 O_z}\]
 as $|\cup_{t=1}^{K_{|A|}} Z^A_t \setminus R^A_t| \le
 \gth(|A|)-|A|$. So the assertion follows from the definition of $\sA$.

 \noindent
 (2). We can imitate the argument of Theorem 3.1.2 of \cite{D07} to
 see that under $\pr_{4,\vI,\vO}$ the number of self-loops and
 multiple edges are asymptotically independent, and both of them
 have asymptotic Poisson distribution whose mean is a function of the moments $\sum_{k,l} k^il^j \IO p_{k,l}, i, j \in \{0, 1, 2\}$.
 So $\pr_{4,\vI,\vO}(F_n)$ has a positive limit. This together with
 the fact that
 \[ \left(Y^x_i, 1\le i\le I_x, x \in [n]\right) \eqd
  \left(\left.\tY^x_i, 1\le i\le I_x, x \in [n]\right| \sum_{x=1}^n\sum_{i=1}^{I_x} \mathbf 1_{\{\tY^x_i = z\}} =O_z \forall z \in [n]\right).\]
 gives the desired inequality.

 \noindent
 (3). For a vector of positive integers $\vy$, we write $X_i(\vy)$ for
 the number of components of $\vy$ which are $i$. We also write
 $\tilde\vY=(\tY_1, \ldots, \tY_N)$ and $\tilde\vY_{a:b}=(\tY_a,
 \tY_ {a+1}, \ldots, \tY_b)$.

 For any and $\vy \in [n]^{\eps N}$,
 \beq
 \frac{P\left(\tilde Y_{1:\eps N}=\vy\left| X_z(\tilde\vy) = N\ga_z
 \forall z\in[n]\right.\right)}{P\left(\tilde Y_{1:\eps
 N}=\vy\right)}
  = \frac{P\left(X_z(\tilde Y_{(\eps N + 1):N}) =
 N\ga_z-X_z(\vy) \forall z\in[n]\right)}{P\left(X_z(\tilde Y) = N\ga_z
 \forall z\in[n]\right)}. \label{fracbd1}
 \eeq
 In order to bound the fraction in \eqref{fracbd1} recall that
 \[ Multinomial(N; \ga_1, \ldots, \ga_n) \eqd \left(Y_1, \ldots,
 Y_n\left| \sum_{i=1}^n Y_i = N\right.\right), \]
 where $\{Y_i\}_{i=1}^n$ are independent and $Y_i \sim
 Poisson(N\ga_i)$. In that case, $P(\sum_{i=1}^n Y_i = N) =
 (1+o(1))/\sqrt{2\pi N}$ by Stirling's formula. So, the ratio in \eqref{fracbd1} is
 \beq \label{fracbd2}
  (1+o(1)) \sqrt{1-\eps} \prod_{z\in [n]}
 \frac{P(Y_z=N\ga_z)}{P(Z_z=N\ga_z)}, \eeq
 where $Y_i \sim Poison(N(1-\eps)\ga_i), Z_i \sim Poison(N\ga_i),
 i\in[n],$ and they are independent. The expression in the last
 display equals
 \begin{align}
 & (1+o(1))\sqrt{1-\eps} \prod_{z=1}^n \frac{(N\ga_z)!}{e^{-N\ga_z} (N\ga_z)^{N\ga_z}} \cdot \frac{e^{-(N\ga_z(1-\eps))} (N\ga_z(1-\eps))^{(N\ga_z-X_z(\vy))}}{(N\ga_z-X_z(\vy))!} \notag \\
 & = (1+o(1))\sqrt{1-\eps} [e^\eps(1-\eps)]^N \exp\left(-log(1-\eps) \eps N\right)\prod_{z=1}^n \prod_{i=1}^{X_z(\vy)} \left(1-\frac{i-1}{N\ga_z}\right). \label{fracbd3}
 \end{align}
Using the inequality $1-\eps \le e^\eps$ we get the desired bound.

%
%

 \noindent
 (4). Using the bound of part (3), the total variation distance between these two measures is
 \begin{align}
 \frac 12 \sum_{\vy \in [n]^{\eps N}} P\left(\tilde \vY_{1:\eps N}=\vy\right)\left|
 \frac{P\left(\tilde \vY_{1:\eps N}=\vy\left| X_z(\tilde\vY_{1:N}) = N\ga_z
 \forall z\in[n]\right.\right)}{
 P\left(\tilde \vY_{1:\eps N}=\vy\right)}-1\right| \notag \\
 \le \exp(-\eps\log(1-\eps) N) P\left(\left(\tilde\vY_{1:\eps N}\right) \in A^c\right)
 + \sup_{\vy \in A} \left|
 \frac{P\left(\tilde\vY_{1:\eps N}=\vy\left| X_z(\tilde\vY_{1:N}) = N\ga_z
 \forall z\in[n]\right.\right)}{
 P\left(\tilde \vY_{1:\eps N}=\vy\right)}-1\right| \label{fracbd4}
 \end{align}
 for any set $A$. Now recall from \eqref{fracbd3} that
 \beq \frac{P\left(\tilde Y_{1:\eps N}=\vy\left| X_z(\tilde\vy) = N\ga_z
 \forall z\in[n]\right.\right)}{P\left(\tilde Y_{1:\eps
 N}=\vy\right)}
 = (1+o(1))\sqrt{1-\eps} [e^\eps(1-\eps)]^N \exp\left(-log(1-\eps) \eps N\right)\prod_{z=1}^n \prod_{i=1}^{X_z(\vy)}
 \left(1-\frac{i-1}{N\ga_z}\right)
 \label{fracbd5} \eeq
 Since $e^\eps \ge 1-\eps$, the first term in the right hand side of \eqref{fracbd5} lies between $1-\eps^2N$ and 1, whereas the second term lies between
 $\exp(\eps^2N)$ and $\exp(2\eps^2 N)$ when $\eps>0$ is small. Also the product term in \eqref{fracbd5} lies between 1 and
 \[ 1-\sum_{z=1}^n \sum_{i=1}^{X_z(\vy)} \frac{i-1}{N\ga_z} = 1 - \sum_{z=1}^n \frac{X_z(\vy)(X_z(\vy)-1)}{2N\ga_z}.\]
 Consequently, if we take
 \[ A_\eta:=\left\{\vy \in [n]^{\eps N}: \sum_{z=1}^n \frac{X_z(\vy)(X_z(\vy)-1)}{2N\ga_z} < \eta\eps^2 N\right\},\]
 then
 \[ \left|
 \frac{P\left(\tilde\vY_{1:\eps N}=\vy\left| X_z(\tilde\vY_{1:N}) = N\ga_z
 \forall z\in[n]\right.\right)}{
 P\left(\tilde \vY_{1:\eps N}=\vy\right)}-1\right|
 = O(\eps^2 N)\]
 whenever $\vy \in A_\eta$. So, in view of \eqref{fracbd4}, it suffices to show that
 $P(\tilde\vY_{1:\eps N} \in A_\eta^c)=O(\exp(-C\eps N))$ for some constant $C>0$ and for some suitable choice of $\eta$.

 Note that the joint distribution of $\{X_z(\tilde\vY_{1:\eps N})\}_{z\in [n]}$ is $Multunomial(\eps N; \ga_1, \ldots, \ga_n)$, so using \eqref{mult} and local central limit theorem
 \[ P\left(\tilde Y_{1:\eps N} \in A_\eta^c\right) = (1+o(1)) \sqrt{2\pi N} P\left( \sum_{i=1}^n \frac{\hat Y_i(\hat Y_i-1)}{N\ga_i} \ge \eta\eps^2 N\right),\]
 where $\{\hat Y_i\}_{i=1}^n$ are independent and $\hat Y_i \sim Poisson(\eps N\ga_i)$. Hence using standard large deviation argument,
 the above probability is at most
 $\exp(-C(\eta)\eps N)$ for some constant $C(\eta)$ such that $C(\eta)>0$ when $\eta$ is large enough.
 This completes the argument.
 \end{proof}

 \begin{rem}
 The assertions (3) and (4) of Lemma \ref{RBN3} are
 still true if $(\tI_1, \ldots, \tI_{\eps n})$ is replaced by
 $(\tI_{i_1}, \ldots, \tI_{i_{\eps n}})$ for some
 (possibly random) index set $\{i_1, i_2, \ldots, i_{\eps n}\}$.

 Similarly, the assertions (3) and (4) of Lemma \ref{RBN4} is true
 if we replace the index set $\{1, 2, \ldots, \eps N\}$ by (possibly random) $\{i_1,
 i_2, \ldots, i_{\eps N}\}$.
 \end{rem}

 \section{Ingredients} \label{ingredients}
  In this section, we will state and prove some of the basic lemmas
  which will be required in proving our main results.

 \begin{lem} \label{laplace transbd}
  Let $X$ be any nonnegative random variable such that $2 (E X)^2 \le E X^2 < \infty$. Then $\log Ee^{-tX} \le var(X) t^2/2 - E(X) t$ for any $t>0$.
 \end{lem}

 \begin{proof}
 Let $\mu= E X$ and $\mu_2= \sqrt{E X^2}$ so that $\gs^2=var(X)=\mu_2^2-\mu^2$. We choose $p=\mu^2/\mu_2^2$ and $\ga=\mu_2^2/\mu$ so that $Y:=(1-p)\mvgd_0+p\mvgd_\ga$ satisfies $EY=\mu$ and $E Y^2=\mu_2^2$. By Benette's inequality \cite{B62},
 \beq\label{Ben ineq} \text{ for any $t>0$, } \log Ee^{-t X} \le \log Ee^{-tY} = \log[(1-p) + p e^{-\ga t}] =: \gf(t). \eeq
  Differentiating  the function $\gf$ and noting that $p\ga=\mu$ and $\mu\ga=\mu_2^2$ we get
  \[ \gf'(t) = \frac{-p\ga e^{-\ga t}}{(1-p) + pe^{-\ga t}} = \frac{-\mu}{(1-p)e^{\ga t} + p}, \quad \gf''(t)=\gs^2 \frac{e^{\ga t}}{[(1-p) e^{\ga t} + p]^2}. \]
 Also note that the quadratic function $f(x)=[(1-p) x+p]^2-x$ has nonnegative slope at $x=1$ if $2(1-p)\ge 1$, which is true by our hypothesis. So $\gf''(t) \le \gs^2$ for any $t\ge 0$. Finally using Taylor series expansion for the function $\gf$ we see that for any $t>0$,
 \beqax
 \gf(t) & = & \gf(0) +\gf'(0) t+ \gf''(u) t^2/2 \text{ for some $u\in[0,t]$} \\
 &\le & -\mu t +\gs^2 t^2/2.
 \eeqax
 This inequality together with \eqref{Ben ineq} gives the desired result.
 \end{proof}

 \begin{lem}\label{phibd}
 For any $\kappa>0$ and $\gD\ge 1$ the function
 $\phi_{\kappa,\gD}(\gamma):=\gamma[\log(\gD/\gamma)]^\kappa$ is increasing for $\gamma \le \gD e^{-\kappa}$ and decreasing for $\gamma\ge \gD e^{-\kappa}$.
 Hence $\phi_{\kappa,\gD}(\gamma) \le \gD(\kappa/e)^\kappa$.
 \end{lem}

 \begin{proof}
 We get the conclusion using elementary method.
 \end{proof}

 \begin{lem} \label{Binldp}
 For $\gd>0, \gth(m):=\gb_1m[\log(n/m)]^{\gb_2}, 0 < \gc \le 1$ and any integer $\gL \ge 1$ there is an $\ep_{\sref{Binldp}} >0$ depending on $\gL, \gb_1, \gb_2, \gc$ such that $m\le \ep_{\sref{Binldp}} n$ and $M\in\dN$ imply
 \[ P(Binomial(\gL M, (\gth(m)/n)^\gc) > \frac{1}{\gc} (1+\gd) M) \le \exp(-(1+\gd/2)M\log(n/m)). \]
 \end{lem}

 \begin{proof}
 A standard large deviations result for the Binomial distribution,
 see e.g., Lemma 2.8.4 in \cite{D07} implies $P( Binomial(\gL M,q) \ge \gL M r) \le
 \exp(-\gL M H_q(r))$ for any $r > q$, where
 \beq
 H_q(r) := r \log\left( \frac{r}{q} \right) + (1-r) \log \left( \frac{1-r}{1-q} \right). \label{Bldpb}
 \eeq
 When $r=(1+\gd)/(\gc\gL)$, the first term in the large deviation bound
 \eqref{Bldpb} is
 \begin{align*}
 \exp(-\gL M r\log(r/q)) & \le \exp\left( - \frac{1}{\gc}(1+\gd)M\left[
 \log\left(\frac{n}{m}\right)^\gc -\log\frac{\gL\gc\gb_1^\gc}{1+\gd}  -\gb_2\gc\log\log\frac{n}{m}\right] \right)
 \end{align*}
 For the second term in the large deviation
 bound in \eqref{Bldpb} we note that $1/(1-q) > 1$ and
 $(1-r)\log(1-r) \ge -1/e$ by Lemma \ref{phibd} (with $\kappa=\gD=1$),
 and conclude
 \[
 \exp\left(-\gL M(1-r)\log \left( \frac{1-r}{1-q} \right) \right) \le
 \exp\left(-\gL M(1-r)\log(1-r) \right) \le \exp(\gL M/e).
 \]
 Combining the last two estimates
 \begin{align*}
 & P(Binomial(\gL M, (\gth(m)/n)^\gc) > \frac{1}{\gc} (1+\gd) M) \\
 & \le
 \exp\left( - (1+\gd)M \log(n/m) +\gb_4 M +\gb_5 M \log\log\frac{n}{m}\right), \end{align*}
 for constants $\gb_4$ and $\gb_5$.

 Now we choose
 \[\ep_{\sref{Binldp}}:= \max\left\{\ep \in (0, e^{-2\gb_5/\gd}): \ep\left[\log\frac{1}{\ep}\right]^{2\gb_5/\gd} \le \exp(-2\gb_4/\gd)\right\}.\]
 Clearly $\ep_{\sref{Binldp}}>0$ and, in view of Lemma \ref{phibd} with $\kappa=2\gb_5/\gd$ and $\gD=1$,  $m\le \ep_{\sref{Binldp}} n$ implies
 \[ (m/n)\left[\log\frac{n}{m}\right]^{2\gb_5/\gd} \le \ep_{\sref{Binldp}}\left[\log\frac{1}{\ep_{\sref{Binldp}}}\right]^{2\gb_5/\gd} \le \exp(-2\gb_4/\gd),\]
 which in turn implies $\gb_4+\gb_5\log\log[n/m] \le (\gd/2)\log(n/m)$. This completes the proof.
 \end{proof}

 \section{Choice of good graph} \label{good graph}
 For $B \subset A\subset [n]$, recall the definition of the forest $\{Z^A_t\}_{t=0}^{K_{|A|}}$ (as described in Section
 \ref{loc nei}) with associated subsets $\{O^{A,B}_t\}$ of `open'
 sites.   Let $\vp$ be the limiting out-degree distribution for
 $\ola\sG_n$, namely
 \beq \label{pdef} \vp = \begin{cases}
           \I\vp & \text{ for } \RBN^2 \\
           \left\{(\I r)^{-1} \sum_{l} l \IO p_{k,l}\right\}_{k=2}^\infty & \text{ for } \RBN^4
           \end{cases}\eeq
 with $p_0=p_1=0$ and mean $r=\sum_k kp_k>2$.

 \begin{prop}\label{cG1}
 There are constants $c_1, c_2, \gd > 0$ such that for $q \in(1/r,
 1/2), K_m=c_1\log_2(c_2\log(n/m))$ and  for $A \subset [n]$ if
 \[ E_A := \cap_{B \in \{B\subset A: |B| \ge (1-\gd)|A|\}} \left\{|O^{A,B}_{K_{|A|}}| \ge (4/\gd)
 q^{-K_{|A|}}|B|\right\},\]
 then there is an $\ep_{\sref{cG1}}>0$ such that for any $a>0$ the probability of
 \[ \cG^1_n :=\cap_{A\in\{A\subset [n]: (\log n)^a \le |A| \le \ep_{\sref{cG1}} n\}} E_A\]
 under $\pr_{i,n}, i=2, 3, 4,$ is $1-o(1)$.
 \end{prop}

 \begin{proof}
 Let $\vp_n=\{p_{n,k}\}$ be the distribution
 \beq \label{p_ndef} \vp^n =\begin{cases}
        \frac 1n \sum_{z\in[n]} \mvgd_{I_z} & \text{ for $\RBN^2$} \\
        \sum_{z\in[n]} O_z \mvgd_{I_z}/\sum_{z\in[n]} O_z  & \text{ for $\RBN^4$}
        \end{cases}.
 \eeq
 In view of Lemma \ref{RBN2}, \ref{RBN3} and \ref{RBN4}, $\vp_n$ approximates
 out-degree distribution for the graph $\ola \sG_n$ with $p_{n,0}=p_{n,1}=0$. Let $r_n:=\sum_k
 kp_{n,k} \in (2,\infty)$ be the mean of $\vp_n$. It is easy to see that
 $r_n \to r$.

 In this proof, we write $\pr$ for the probability distribution on
 the forests $\{Z^A_t\}_{t\ge 0, A\subset [n]}$, when $\vp_n$ is used
 as its offspring distribution. We also use $\tilde\pr$ as a dummy
 replacement for $\pr_{2,\vI}, \pr_{3,\vO}$ and $\pr_{4,\vI,\vO}(\cdot| F_n)$.
 Lemma \ref{RBN2}, \ref{RBN3} and \ref{RBN4} suggest that for any
 event $F$ involving the structure of the graph which depends on at
 most $\eps n$ many vertices of the graph,
 \beq \label{compare}
 \left.
 \begin{array}{c}
   \I\vp_{\otimes n}\left(\left\{\vI: \tilde\pr(F) \le C_{\sref{RBN2}} \pr(F)\right\}\right) \\
   \\
   \OU\vp_{\otimes n}\left(\left\{\vO: \tilde\pr(F) \le C_{\sref{RBN3}}(1-\eps)^{-1}(1+o(1)) \pr(F)\right\}\right) \\
   \\
   \IO\vp_{\otimes n}\left(\left.\left\{(\vI,\vO): \tilde\pr(F) \le C_{\sref{RBN4}} \exp(-\OU r\eps\log(1-\eps)n) \pr(F)\right\}\right| E_n\right)
 \end{array}
 \right\} = 1-o(1).\eeq

 Now fix $\eta \in (0, 1-1/qr)$, and
 \[ \gc := \begin{cases}
 1 & \text{ for $\RBN^2$}\\
 \frac{\ga-2}{\ga-1} & \text{ for $\RBN^3$ and $\RBN^4$ when $\OU p_k \sim ck^{-\ga}$ and $\ga>3$}
 \end{cases}.
 \]
 Clearly $1/\gc<2<r(1-\eta)$, so we can choose $\gd \in (0, 1/10)$ such that $(1+5\gd)/(2\gc) < 1$.
 We need to introduce some more notations, let
 \begin{align}
 \tr & := r_n(1-\eta) \text{ so that  $q\tr>1$ for large enough $n$}, \notag\\
 \rho & > 0 \text{ be such that }
 2^{\rho-1}\left(1-\frac{1+5\gd}{2\gc}\right) \ge 1 \text{ and }
 (q\tr)^\rho\left(1-\frac{1+5\gd}{\gc \tr}\right) > 1,\notag\\
 \gs &\ge 1 \text{ be such that }
 \left[(q\tr)^\rho\left(1-\frac{1+5\gd}{\gc \tr}\right)\right]^\gs \ge
 r_n^\rho,\notag\\
 I_n(\eta) &:= \sup_\theta\left(\theta r(1--\eta)-\log\left(\sum_k e^{\theta k}
 p_{n,k}\right)\right) > 0 \text{ be the large deviation rate function for $\vp$ } \notag\\
 k_m &:= \log_2\left[\frac{1+3\gd}{(\tr-\gc^{-1}(1+5\gd)) I_n(\eta)} \log\frac{n}{m}\right] \text{ and $K_m=\rho\gs k_m$ for $m \le
 n$, so that } \label{K_mdef} \\
 \gth(m) & := m+\sum_{l=1}^{K_m} (2r)^l \le
 \gb_1m[\log(n/m)]^{\gb_2} \notag \\
 \text{ for } \gb_1 & =\left(1+\frac{2r}{2r-1}\right)
 \left(\frac{1+3\gd}{(\tr-\gc^{-1}(1+5\gd)) I_n(\eta)}\right)^{\rho\gs\log_2(2r)} \text{ and
 } \gb_2 := \rho\gs\log_2(2r). \notag
 \end{align}

 Suppose $B\subset A\subset [n]$ are subsets such that $|A|=m$ and $|B|\ge (1-\gd)m$. For $k\ge 1$, define the events
 \beqax
 H^{A,B}_k &:=&  \left\{\sum_{i=1}^{\rho} |C^{A,B}_{\rho(k-1)+i}| \le \frac{1}{\gc}(1+5\gd)|O^{A,B}_{\rho(k-1)}|\right\},\\
 L^{A,B}_{k,j} &:=& \left\{|Z^{A,B}_{\rho(k-1)+j}| \ge \tr |O^{A,B}_{\rho(k-1)+j-1}|\right\} \text{ and } L^{A,B}_k :=  \cap_{j=1}^{\rho} L^{A,B}_{k,j}. \eeqax
 Note that on the event $H^{A,B}_k$,
 \beqa
 |O^{A,B}_{\rho k}| &\ge&  2|O^{A,B}_{\rho k-1}| - |C^{A,B}_{\rho k}|  \notag\\
 &\ge&  2^2|O^{A,B}_{\rho k-2}| - 2|C^{A,B}_{\rho k-1}| -|C^{A,B}_{\rho k}|  \notag\\
 &\ge& \cdots  \notag\\
 &\ge&  2^\rho|O^{A,B}_{\rho(k-1)}| - \sum_{i=1}^\rho 2^{\rho-i} |C^{A,B}_{\rho(k-1)+i}|  \notag\\
 &\ge&  2^\rho|O^{A,B}_{\rho(k-1)}| - 2^{\rho-1}\sum_{i=1}^\rho |C^{A,B}_{\rho(k-1)+i}|  \notag\\
 &\ge& (2^\rho-2^{\rho-1}\gc^{-1} (1+5\gd)) |O^{A,B}_{\rho(k-1)}| \ge
 2|O^{A,B}_{\rho(k-1)}|, \label{Oineq1}
 \eeqa
 by the choice of $\rho$.
 Since $|O^{A,B}_t| \le (2r) |O^{A,B}_{t-1}|$ for any $t\ge 1$,
 a similar argument which leads to the previous display suggests
 that the following inequalities are true on the event $H^{A,B}_k\cap \cap_{j=1}^i
 L^{A,B}_{k,j}$.
 \beqa
 |O^{A,B}_{\rho(k-1)+i}|
 & \ge & \tr |O^{A,B}_{\rho(k-1)+i-1}| - |C^{A,B}_{\rho(k-1)+i}|  \notag\\
 & \ge & \tr^2 |O^{A,B}_{\rho(k-1)+i-2}| -\tr |C^{A,B}_{\rho(k-1)+i-1}| - |C^{A,B}_{\rho(k-1)+i}| \ge \cdots  \notag\\
 &\ge&  \tr^i |O^{A,B}_{\rho(k-1)}| - \sum_{j=1}^i \tr^{i-j} |C^{A,B}_{\rho(k-1)+j}|  \notag\\
 &\ge&  \tr^i |O^{A,B}_{\rho(k-1)}| - \tr^{i-1} \sum_{j=1}^i |C^{A,B}_{\rho(k-1)+j}| \notag \\
 &\ge& (\tr^i-\tr^{i-1} \gc^{-1} (1+5\gd)) |O^{A,B}_{\rho(k-1)}| \label{Oineq2} \\
 & \ge &
 \tr\left(1-\frac{1+5\gd}{\gc \tr}\right) |O^{A,B}_{\rho(k-1)}|, \label{Oineq3}
 \eeqa
 Taking $i=\rho$ in \eqref{Oineq2},
 \beq\label{Oineq4}
 |O^{A,B}_{\rho k}| \ge  \tr^\rho\left(1-\frac{1+5\gd}{\gc\tr}\right) |O^{A,B}_{\rho(k-1)}|
 \text{ on the event } H^{A,B}_k \cap L^{A,B}_k.\eeq
 Recalling $K_m=\rho\gs k_m$ and using \eqref{Oineq1} and \eqref{Oineq4} repeatedly,
 \[ |O^{A,B}_{K_m}| \ge |B|2^{k_m}\left[\tr^\rho\left(1-\frac{1+5\gd}{\gc\tr}\right)\right]^{(\gs-1)k_m} \text{ on the event
 $\cap_{k=1}^{\gs k_m} H^{A,B}_k \cap \cap_{k=k_m+1}^{\gs k_m} L^{A,B}_k$.}\]
 Now note that
 \[ q^{\rho\gs} \left[\tr^\rho\left(1-\frac{1+5\gd}{\gc\tr}\right)\right]^{\gs-1}
 \ge \left[(q\tr)^\rho\left(1-\frac{1+5\gd}{\gc\tr}\right)\right]^\gs r^{-\rho} \ge 1\]
 by the choices of $\rho$ and $\gs$. So
 if
 \beq \label{mnbd1}
 (m/n) \le \exp\left(-\frac{4(\tr-\gc^{-1}(1+5\gd)) I_n(\eta)}{\gd(1+3\gd)}\right) \text{ so that $2^{k_m}\ge (4/\gd)$},\eeq
 then
 \beq\label{Oineq5}
 |O^{A,B}_{K_m}| \ge |B|(4/\gd)q^{-K_m}
 \text{on the event } \cap_{k=1}^{\gs k_m} H^{A,B}_k \cap \cap_{k=k_m+1}^{\gs k_m}
 L^{A,B}_k.\eeq
 To estimate the probability of the event in \eqref{Oineq5} recall
 that $|C^{A,B}_{t+1}|, |O^{A,B}_{t+1}| \le (2r) |O^{A,B}_t|$ for any $t\ge 0$ and by Lemma \ref{collision} each site is included
 in $C^{A,B}_t \subset C^A_t$ with probability at most $c(\gth(m)/n)^\gc$. So
 for any $k\ge 1$, $\sum_{i=1}^\rho |C^{A,B}_{\rho(k-1)+i}|$ conditionally on $|O^{A,B}_{\rho(k-1)}|$
 is stochastically dominated by the $Binomial(\gL M, c(\gth(m)/n)^\gc)$ distribution, where
 $\gL=(2r) + (2r)^2 + \cdots + (2r)^\rho$ and $M=|O^{A,B}_{\rho(k-1)}|$. Hence, applying Lemma
 \ref{Binldp} with the above choices of $\gL$ and $M$ if
 \beq \label{mnbd2} m \le \ep_{\sref{Binldp}}(\gL,5\gd,\eta,\gc) n,\eeq
 then
 \[
 \pr\left(\left.(H^{A,B}_k)^c\right| |O^{A,B}_{\rho(k-1)}|\right)
 \le \exp\left(-(1+5\gd/2)|O^{A,B}_{\rho(k-1)}| \log(n/m)\right).\]
 Since $|O^{A,B}_{\rho(k-1)}| \ge |B|$ on the event $\cap_{j=1}^{k-1} H^{A,B}_j$ by
 \eqref{Oineq1},  the above inequality reduces to
 \beq
 \pr\left((H^{A,B}_k)^c \cap\cap_{j=1}^{k-1} H^{A,B}_j\right)
 \le \exp\left(-(1+5\gd/2)|B|\log(n/m)\right) \le \exp\left(-(1+\gd)m\log(n/m)\right). \label{err1}
 \eeq
 The last inequality follows from the fact that $|B|\ge(1-\gd)m$ and
 $\gd\in (0,1/10)$, which makes $(1+5\gd/2)(1-\gd)\ge1+\gd$.

 By the choice of $I_n(\eta)$, a standard large deviation argument for
 the sum of \iid random variables yields
 \[\pr\left(\left.(L^{A,B}_{k,i})^c\right| |O^{A,B}_{\rho(k-1)+i-1}|\right)
 \le \exp\left(-|O^{A,B}_{\rho(k-1)+i-1}| I_n(\eta)\right)\]
 for any $k\ge 1$ and $1\le i\le \rho$.
 Now repeated applications of the inequality in
 \eqref{Oineq1} suggest that $|O^{A,B}_{\rho k_m}| \ge
 2^{k_m}|B|$ on the event $\cap_{j=1}^{k_m} H^{A,B}_j$. In view of
 \eqref{Oineq1} and \eqref{Oineq3}, for any
 $k>k_m$ and $1\le i\le \rho$,
 \[ |O^{A,B}_{\rho(k-1)+i-1}| \ge (\tr-\gc^{-1}(1+5\gd))|O^{A,B}_{\rho(k-1)}|
 \ge (\tr-\gc^{-1}(1+5\gd)) |O^{A,B}_{\rho k_m}|\]
 on the event $\cap_{j=k_m+1}^k H^{A,B}_j \cap_{j=1}^{i-1} L^{A,B}_{k,j}$.
 So the inequality in the last display reduces to
 \beqa
 \pr\left((L^{A,B}_{k,i})^c\cap_{j=1}^{i-1} L^{A,B}_{k,j}
 \cap_{j=1}^{k} H^{A,B}_j\right)
 & \le & \exp\left(-(\tr-\gc^{-1}(1+5\gd))2^{k_m}|B| I_n(\eta)\right) \notag\\
 & \le & \exp(-(1+3\gd)|B|\log(n/m)) \notag \\
 & \le & \exp(-(1+\gd)m\log(n/m)). \label{err2}
 \eeqa
 The last two inequalities follow from the definition of
 $k_m$ and the facts that $|B| \ge (1-\gd)m$, which implies $(1+3\gd)|B| \ge (1+\gd)m$ for $\gd \in (0,1/10)$.
 Applying Lemma \ref{phibd} with $\kappa=\gD=1$,
 \beq \label{mnbd3}
 m\log(n/m) = n \phi_{1,1}(m/n) \ge n \phi_{1,1}(1/n) = \log n \text{ for $m\le n/e$.}\eeq
 Combining \eqref{Oineq5}, \eqref{err1} and \eqref{err2} if $m/n$ is small satisfying \eqref{mnbd1}, \eqref{mnbd2} and \eqref{mnbd3}, then
 \beqax
 \pr\left(|O^{A,B}_{K_m}| < |B|(4/\gd)q^{-K_m}\right)
 &\le & \pr\left(\left(\cap_{k=1}^{\gs k_m} H^{A,B}_k \cap \cap_{k=k_m+1}^{\gs k_m}
 L^{A,B}_k\right)^c\right) \\
 &\le & \sum_{k=1}^{\gs k_m} \pr\left((H^{A,B}_k)^c \cap_{j=1}^{k-1}
 H^{A,B}_j\right) \\
 && \hspace{-2cm} + \sum_{k=k_m+1}^{\gs k_m} \sum_{i=1}^{\rho} \pr\left((L^{A,B}_{k,i})^c\cap_{j=1}^{i-1} L^{A,B}_{k,j}
 \cap_{j=1}^{k} H^{A,B}_j\right)\\
 &\le & [\gs+\rho(\gs-1)] k_m \exp(-(1+\gd)m\log(n/m)) \\
 &\le & [\gs+\rho(\gs-1)] \log_2[C\log n] \\ && \exp(-(1+3\gd/4)m\log(n/m))n^{-\gd/4} \\
 &\le & \exp(-(1+3\gd/4)m\log(n/m))
 \eeqax
 for large enough $n$. Since the event considered in the last display involves
 at most $\gth(m)$ vertices of the graph, the above estimate together with
 \eqref{compare} implies
 \[ \tilde \pr\left(|O^{A,B}_{K_m}| < |B|(4/\gd)q^{-K_m}\right)
 \le \exp(-(1+3\gd/8)m\log(n/m)),\]
 with ($\I\vp_{\otimes n}/\OU\vp_{\otimes n}/\IO\vp_{\otimes n}$) probability $1-o(1)$ provided $m/n \le \eps$ is small.

 Using this estimate and union bound we see that if $m/n$ is small, then
 \begin{align}
 & \tilde\pr \left(\cup_{A\in\{A\subset [n]: |A|=m\}} E_A^c\right)  \notag \\
 & \le \tilde\pr \left(\cup_{m'\in[(1-\gd)m, m]} \cup_{\{(A,B): B\subset A\subset [n], |A|=m, |B|=m'\}} \left\{|O^{A,B}_{K_m}|<(4/\gd)q^{-K_m}|B|\right\}\right)  \notag \\
 & \le \sum_{m'\in[(1-\gd)m, m]} {n\choose m} {m\choose
 m'}\exp\left(-(1+3\gd/8)m\log\frac{n}{m}\right). \label{cGbd1}
 \end{align}
 It is easy to check that ${L\choose l}\le \frac{L^l}{l!} \le (Le/l)^l$ for any positive integers $l\le L$ and the function $\phi_{1,e}(\cdot)$ defined in Lemma \ref{phibd} is increasing on $(0,1)$. So for $m'\ge (1-\gd)m$,
 \beqax
 {n\choose m} \le \left(\frac{ne}{m}\right)^m \text{ and } {m\choose m'}
 & =   {m\choose m-m'} & \\
 & \le   \left(\frac{me}{m-m'}\right)^{m-m'}  & =\exp\left[ m \phi_{1,e}\left(\frac{m-m'}{m}\right)\right]\\
 & \le  \exp(m\phi_{1,e}(\gd)) & \le (e/\gd)^{\gd m}.
 \eeqax
 Also there are at most $m\le e^m$ choices for $m'$.
 Using these bounds the right hand side of \eqref{cGbd1}  is
 \beqax
 &\le & \exp[m+m\log(ne/m)+m\gd\log(e/\gd)-(1+3\gd/8)m\log(n/m)]\\
 &\le & \exp[-(3\gd/8)m\log(n/m)+\gD_1 m]
 \eeqax
 for some constant $\gD_1$. If $m/n \le \exp (-8\gD_1/\gd)$, then the right hand side of the last display is  $\le \exp[-(\gd/4)m\log(n/m)]$. Therefore, if
 $\ep_{\sref{cG1}}$ is chosen small enough, then for any $m\le \ep_{\sref{cG1}} n$,
 \[ \tilde\pr \left(\cup_{A\in\{A\subset [n]: |A|=m\}} E_A^c\right) \le \exp[-(\gd/4)m\log(n/m)].\]
 Combining this with the fact that $m\mapsto m\log(n/m)$ is
 increasing for $m\le n/e$ (by Lemma \ref{phibd}),
 \beqax
 \tilde\pr((\cG^1_n)^c)
 & \le  & \sum_{m \le [(\log n)^a,\ep_{\sref{cG1}} n]} \tilde\pr \left(\cup_{A\in\{A\subset [n]: |A|=m\}} E_A^c\right)
 \le \sum_{m \in [(\log n)^a,\ep_{\sref{cG1}} n]} \exp[-(\gd/4)(\log n)^a\log(n/(\log n)^a)] \\
 & \le &  n \exp[-(\gd/4)(\log n)^{1+a}(1+o(1))] = o(1/\sqrt
 n). \eeqax
 This together with \eqref{compare} completes the proof.
 \end{proof}

 Recall the definition of $\mvpi(\cdot,\cdot)$ from \eqref{pidef} and let $CL^x_k := \cup_{l=0}^k\ola Z^{\{x\}}_l$ be the oriented
 cluster of depth $k$ starting from $x\in[n]$ in the graph
 $\ola\sG_n$.
 Now define the events
 \beqax
 A_x &:=& \left\{\left|\ola\xi^{\{x\}}_{2a\log\log n/\log(q\tr)}\right| \ge (\log n)^a
 \right\}, \\
 A_{x,y} &:=& \left\{CL^x_{2a\log\log n/\log(q\tr)}
 \cap CL^y_{2a\log\log n/\log(q\tr)}=\emptyset\right\}.\eeqax

 \begin{prop} \label{cG2}
 For $\mvpi(\cdot,\cdot)$ as in \eqref{pidef}, $\vp$ as in \eqref{pdef} and $\eps>0$ let
 $\pi:=\mvpi(\vp,q)$ and
 \[ \cG^2_n := \left\{\sG_n: \frac 1n \sum_{x\in [n]} P_{\sG_n,q}(A_x) \ge \pi-\eps\right\} \cap
 \left\{\sum_{x,y\in[n], x\ne y}\mathbf 1_{A_{x,y}^c} \le n^{9/5}\right\}. \]
 Then $\pr_{i,n}(\cG^2_n)=1-o(1)$ for $i=2,3,4$.
 \end{prop}

 \begin{proof}
 In this proof also the notations $\pr$ and $\tilde\pr$ serve the same purpose as they did in the proof of Proposition \ref{cG1}. $\bE$ and $\tilde\bE$ denote
 the corresponding expectations.

 First we note that if
 \beq \label{clustbd}
 C_x := \left\{\left|CL^x_{2a\log\log n/\log(q\tr)}\right| \le
 n^{1/4}\right\}, \text{ then } \pr(C_x), \tilde \pr(C_x) \ge 1 - n^{-1/8},\eeq
 using Markov inequality. This bound together with (4) of Lemma \ref{RBN3} and
 \ref{RBN4} with $\eps=n^{-3/4}$ implies
 \beq \label{TVbd}
 \left|\bE\left(P_{\sG_n,q}(A_x)\right) - \tilde \bE\left(P_{\sG_n,q}(A_x)\right)\right|
 =o(1)+\left|\bE\left(P_{\sG_n,q}(A_x)\right)\mathbf 1_{C_x} - \tilde \bE\left(P_{\sG_n,q}(A_x)\right)\mathbf
 1_{C_x}\right|=o(1). \eeq
 Now if $B_x$ denotes the event that collision does not occur in the cluster $CL^x_{2a\log\log n/\log(q\tr)}$,
 then combining \eqref{clustbd} with Lemma \ref{RBN2}, \ref{RBN3} and \ref{RBN4},
 \beq \label{B_x}
 \pr(B_x^c) = o(1)+\pr(B_x^c \cap C_x) = o(1) + \tilde\pr(B_x^c \cap C_x) = o(1).
 \eeq
 On the event $B_x$, the law of $|\ola\xi^{\{x\}}_t|, 0\le t\le 2a\log\log n/\log(q\tr)$ under the annealed measure $\pr\times P_{\sG_n,q}$ is same as that of a branching process with offspring distribution $(1-q) \mvgd_0+q\vp_n$, where $\vp_n$ is as in \eqref{p_ndef}. So if $\{Z_t\}_{t\ge 0}$ is such a branching process with $Z_0=1$, then using \eqref{B_x},
 \[ \bE[P_{\sG_n,q}(A_x)] = P\left(Z_{2a\log\log n/\log(q\tr)} >(\log n)^a\right) + o(1).\]
 Imitating a branching process large deviation result (Theorem 3 in \cite{A94}) and following the argument which leads to Lemma 2.2 in \cite{CD11}, the above expression is $\ge \pi(\vp_n)-\eps/4+o(1)$.
 This together with \eqref{TVbd} and the fact that $\pi(\vp_n) \to \pi$
 as $n \to \infty$ implies
 \beq \label{A_xbd}
 \tilde\bE[P_{\sG_n,q}(A_x)] \ge \pi - \eps/2\eeq
 if $n$ is large enough. Also using \eqref{clustbd} and Lemma
 \ref{RBN2}, \ref{RBN3} and \ref{RBN4}
 \beq \label{A_xybd}
 \tilde\pr(A_{x,y}^c) \le \tilde\pr(C_x^c) + \tilde\pr(C_y^c) +
 \tilde\pr(A_{x,y}^c \cap C_x \cap C_y) \le 2n^{-1/8} + c
 n^{1/4}\left(\frac{n^{1/4}}{n}\right)^{(\ga-2)/(\ga-1)} \le
 cn^{-1/8} \eeq
 for some constant $c$.
 Using the last inequality and following the argument which leads to (2.13) in \cite{CD11}, if $x_1, x_2 \in [n]$ are such that $x_1\ne x_2$, then
 \beqax
  && \tilde\bE[P_{\sG_n,q}(A_{x_1}) P_{\sG_n,q}(A_{x_2})] - \tilde\bE[P_{\sG_n,q}(A_{x_1})]\tilde\bE[P_{\sG_n,q}(A_{x_2})] \\
 & \le & \tilde\pr(A_{x_1,x_2}^c)[1+1/\tilde\pr(A_{x_1,x_2})] \le
 cn^{-1/8}.
 \eeqax
 So using a standard second moment argument and then combining with
 \eqref{A_xbd}
 \[
 \tilde\pr\left(\sum_{x\in[n]} P_{\sG_n,q}(A_x) < n(\pi-\eps)\right) =
 o(1).
 \]
 By a similar argument if $x_1, x_2, x_3, x_4 \in [n]$ are such that
 $x_1\ne x_2$, $x_3\ne x_4$ and $\{x_1, x_2\} \cap \{x_3, x_4\} =
 \emptyset$, then
 \beqax
 && \tilde\bE[A_{x_1,x_2} \cap A_{x_3,x_4}] - \tilde\bE[A_{x_1,x_2}]\tilde\bE[A_{x_3,x_4}] \\
 & \le & \tilde\pr(\cup_{i\in\{1,2\}, j\in\{3,4\}} A_{x_i,x_j}^c)[1+1/\tilde\pr(\cap_{i\in\{1,2\}, j\in\{3,4\}} A_{x_i,x_j})] \le c n^{-1/8} \le c
 n^{-1/8},
 \eeqax
 and hence combining with \eqref{A_xybd} and using the standard
 second moment argument,
 \[ \tilde\pr\left(\sum_{x,y\in[n], x\ne y} \mathbf 1_{A_{x,y}^c} >  n^{-1/10}{n\choose 2}\right) =
 o(1).\]
 This completes the proof.
 \end{proof}

 \section{Proofs of the Theorems} \label{proof}
 Let $\sF$ be a forest consisting of $m$ rooted
 directed trees and let $\sF_{k,i}$ denote the set of vertices of
 the $i$-th tree which are at oriented distance $k$ from the root
 level and $\sF_k=\cup_{i=1}^m \sF_{k,i}$.

 \begin{lem} \label{forest est}
 If $P_{\sF, q}$ denotes the law of $\{\ola\xi^A_t, A\subset \sF_0, t\ge 0\}$ on the directed
 forest $\sF$ and if $|\sF_k| \ge 2q^{-k}|\sF_0|$, then
 \[P_{\sF, q}\left( \left| \ola\xi^{\sF_0}_k\right| \le |\sF_0| \right) \le \exp\left( -c q^k |\sF_k|^2/\sum_{i=1}^m |\sF_{k,i}|^2\right). \]
 \end{lem}

 \begin{proof}
 Let $m=|\sF_0|$. For $x\in \sF_k$ let $Y_x:=\mathbf 1\{x \in \ola\xi^{\sF_0}_k\}$ and for $1\le i\le m$ let $N_i:=\sum_x Y_x \mathbf 1\{x\in \sF_{k,i}\}$. It is easy to see that if $l(x,y)$ equals half of the distance between $x$ and $y$ in the forest ignoring the orientation of the edges, then
 \begin{align*}
 & E_{\sF, q} Y_x = q^k,  \quad
 E_{\sF, q} (Y_x Y_z) = \begin{cases} q^k & \text{ if } x=z\\
                       q^{k+l(x,y)-1} & \text{ if } 1\le l(x,y) \le k,\\
                       q^{2k} & \text{ otherwise }  \end{cases} \text{ so that } \\
 & E_{\sF, q} N_i = q^k|\sF_{k,i}|,  E_{\sF, q} N_i^2 = \sum_{x,z \in \sF_{k,i}} E_{\sF, q} (Y_x Y_z) \in \left[q^{2k-1}|\sF_{k,i}|^2 , q^k|\sF_{k,i}|^2\right].
 \end{align*}
  By our hypothesis, $\sum_{i=1}^m E_{\sF, q} N_i \ge 2m$. So applying Lemma \ref{laplace transbd}
 \beqax
 P_{\sF, q}\left(\left|\ola\xi^{\sF_0}_k\right| \le m\right)
 & = & P_{\sF, q}\left(\sum_{i=1}^m N_i \le m\right)
 \le \exp\left(tm + \sum_{i=1}^m \log E_{\sF, q} e^{-t N_i}\right)   \\
 & \le & \exp\left(tm + \sum_{i=1}^m \left[-t E_{\sF, q} N_i + \left(E_{\sF, q}N_i^2 - [E_{\sF, q} N_i]^2\right)t^2/2\right] \right) \\
 & \le & \exp\left(-\sum_{i=1}^m \left[(t/2) E_{\sF, q} N_i - \left(E_{\sF, q}N_i^2-\left[E_{\sF, q}N_i\right]^2\right) t^2/2\right] \right)
 \eeqax
 for any $t\ge 0$. Optimizing the last expression with respect to $t$ and noting that $at-bt^2/2 \le a^2/2b$ for any $a, b>0$ we have
 \beqa \label{perc bd1}
 P_{\sF, q}\left(\left|\ola\xi^{\sF_0}_k\right| \le m\right)
 & \le & \exp\left(-\frac{\left(\sum_{i=1}^m E_{\sF, q} N_i\right)^2}{8\sum_{i=1}^m \left(E_{\sF, q} N_i^2 - (E_{\sF, q} N_i)^2\right)}\right) \\
 & \le & \exp\left(-\frac{q^{2k}/8}{q^k-q^{2k}} \left(\sum_{i=1}^m\left|\sF_{k,i}\right|\right)^2/\sum_{i=1}^m\left|\sF_{k,i}\right|^2\right).
 \eeqa
 \end{proof}

 \begin{prop} \label{Bigsurv}
 Let $\ep_{\sref{cG1}}$ and $\cG^1_n$ be as in Proposition
\ref{cG1} and $K_m$ be as in \eqref{K_mdef}. There are  constants
$C_{\sref{Bigsurv}}, b > 0$ such that if $\sG_n \in \cG^1_n$ and
$A\subset [n]$ has size $m\le \ep_{\sref{cG1}} n$, then
 \[ P_{\sG_n,q}\left(\left|\ola\xi^A_{K_{m}}\right| \le m\right) \le \exp\left(-C_{\sref{Bigsurv}} m(\log(n/m))^{-b}\right).\]
 \end{prop}

 \begin{proof}
 For $\sG_n \in \cG^1_n$, any $A \subset [n]$ with $|A|=m\le
 \ep_{\sref{cG1}} n$ and $\gd$ as in Proposition \ref{cG1}, define
 \[\tau_A := \left\{x\in A: |O^{A,\{x\}}_{K_m}| \ge (4/\gd)q^{-K_m}\right\}.\]
 Clearly $|\tau_A| \ge \gd |A|$, because otherwise $B=A\setminus
 \tau_A$ will have $|B| \ge (1-\gd)|A|$ and
 \[ |O^{A,B}_{K_m}| \le \sum_{x\in B} |O^{A,\{x\}}_{K_m}| <(4/\gd)q^{-K_m} |B|\]
 by the definition of $\tau_A$ and this contradicts the fact that $\sG_n \in \cG^1_n$.

 Let $\sF$ be the subgraph of $\{Z^{A,\tau_A}_t\}_{t=0}^{K_m}$induced by the vertex set
 \[ \cup_{x\in\tau_A} \left(\cup_{i=0}^{K_m-1} O^{A,\{x\}}_t  \cup \Pi(\lceil(4/\gd)q^{-K_m}\rceil, O^{A,\{x\}}_{K_m})\right).\]
 So $\sF$ is a labeled directed forest with depth $K_m$ such that $|\sF_0| \ge \gd m$ and $|\sF_{K_m,i}|=\lceil(4/\gd)q^{-K_m}\rceil$ for all $i$. Applying Lemma \ref{forest est} with $k$ replaced by $K_m$ and $m$ replaced by $\gd m$, and noting that $|\ola \xi^A_{K_m}|$ stochastically dominates $|\ola \xi^{\sF_0}_{K_m}|$,
 \[ P_{\sG_n,q}\left(\left|\ola\xi^A_{K_m}\right| \le m\right) \le \exp\left(-\frac 18 q^{K_m} \gd m\right).\]
 This proves the result.
 \end{proof}

 \begin{proof}[Proof of Theorem \ref{thm:RBN2} and \ref{thm:RBN4}]
 We take $\cG_n:=\cG^1_n \cap \cG^2_n$, where $\cG^1_n$ and $\cG^2_n$ are as in Proposition \ref{cG1} and \ref{cG2} respectively,
 and we will see that
 \[\gD:=\frac 12 C_{\sref{Bigsurv}}\ep_{\sref{cG1}}[\log(1/\ep_{\sref{cG1}})]^{-b}\]
 will suffice, where $C_{\sref{Bigsurv}}, b$ are as in Proposition \ref{Bigsurv}.  Clearly $P_{i,n}(\cG_n)=1-o(1)$.
 Define
 \[ T_x:=\inf\left\{t\ge 1: |\ola\xi^{\{x\}}_t| \ge \ep_{\sref{cG1}} n\right\}.\]
 We take $T_x=\infty$ if $\ola\xi^{\{x\}}_t$ never reaches
 $\ep_{\sref{cG1}} n$.
 Recalling the definition of the event $A_x$ from \eqref{A_xbd} and then applying Proposition \ref{cG1}, $\sG_n \in
 \cG_n$ implies
 \beqa
 && P_{\sG_n,q} \left(A_x \cap \left\{T_x > 2a\log\log n/\log(q\tr) + \sum^{\ep_{\sref{cG1}} n-1}_{m=(\log n)^a} K_m\right\}\right) \notag \\
 &\le& \sum_{m=(\log n)^a}^{\ep_{\sref{cG1}} n}  \exp\left(-C_{\sref{Bigsurv}} m(\log(n/m))^{-b}\right) \notag \\
 &\le& n \exp\left(-C_{\sref{Bigsurv}} (\log n)^a(\log(n/(\log n)^a))^{-b}\right) = o(1/n) \label{xsurv1}
 \eeqa
 if $a$ is large enough. For $i\ge 1$ if $|\ola\xi^{\{x\}}_{T_x +
 (i-1) K_m}| \ge \ep_{\sref{cG1}} n$, then we can again apply
 Proposition \ref{cG1} with $A$ replaced by any subset of $\ola\xi^{\{x\}}_{T_x+(i-1)K_m}$ consisting of $\ep_{\sref{cG1}} n$ many
 vertices to have
 \[ P_{\sG_n,q}\left(\left\{|\ola\xi^{\{x\}}_{T_x + i K_m}| < \ep_{\sref{cG1}} n\right\} \cap \left\{|\ola\xi^{\{x\}}_{T_x+(i-1)K_m}|\ge\ep_{\sref{cG1}} n\right\}\right) \le \exp\left(-C_{\sref{Bigsurv}} \ep_{\sref{cG1}}[\log(1/\ep_{\sref{cG1}})]^{-b}n\right),\]
 which in turn implies
 \beq\label{xsurv2}
 P_{\sG_n,q}\left(\left\{\left|\ola\xi^{\{x\}}_{T_x + e^{\gD n} K_m}\right| < \ep_{\sref{cG1}} n\right\} \cap \{T_x < \infty\}\right) = e^{\gD n}e^{-2\gD n}=o(1/n).\eeq
 Combining \eqref{xsurv1} and \eqref{xsurv2} and using union bound,
 \[P_{\sG_n,q}\left(\cup_{x\in[n]} \left[A_x \cap \left\{\ola\xi^{\{x\}}_{\exp(\gD n)} = \emptyset\right\}\right]\right) \le n o(1/n)=o(1).\]
 This together with the duality relationship between $\xi_t$ and $\ola\xi_t$ suggests
 \beq \label{densitybd}
 P_{\sG_n,q}\left(\xi^{[n]}_{\exp(\gD n)} \supset \{x\in[n]: A_x \text{ occurs}\}\right)
 = P_{\sG_n,q}\left(\ola\xi^{\{x\}}_{\exp(\gD n)} \ne \emptyset \text{ if  $A_x$ occurs}\right) = 1-o(1).
 \eeq
 Now in order to estimate the size of $\{x \in [n]: A_x \text{ occurs}\}$, we will use a second moment argument for $\sum_{x\in[n]} \mathbf 1_{A_x}$. Note that
 \[ E_{\sG_n,q}\left[\sum_{x\in[n]} \mathbf 1_{A_x} - \sum_{x\in[n]} P_{\sG_n,q}(A_x)\right]^2 = \sum_{x,y\in[n]} [P_{\sG_n,q}(A_x \cap A_y) -P_{\sG_n,q}(A_x) P_{\sG_n,q}(A_y)].\]
 Recalling the definition of the event $A_{x,y}$ from \eqref{A_xybd} if $A_{x,y}$ occurs, then the corresponding summand in the above sum is 0, otherwise the summands are at most 1.
 Keeping this observation in mind and using the fact that $\sG_n \in \cG^2_n$,
 \[ E_{\sG_n,q}\left[\sum_{x\in[n]} \mathbf 1_{A_x} - \sum_{x\in[n]} P_{\sG_n,q}(A_x)\right]^2 \le n + \sum_{x,y\in[n], x\ne y} \mathbf 1_{A_{x,y}^c} \le n + {n\choose 2} o(1).\]
 Also from Proposition \ref{cG2} $\sum_{x\in[n]} P_{\sG_n,q}(A_x) \ge n(\pi-\eps)$ for $\sG_n \in \cG^2_n$. Therefore, by Chebyshev inequality
 \[ P_{\sG_n,q}\left(\sum_{x\in[n]} \mathbf 1_{A_x} < (\pi-2\eps) n\right) =o(1).
 \]
 Combining this with \eqref{densitybd}
 \[P_{\sG_n,q}\left(\left|\xi^{[n]}_{\exp(\gD n)}\right| \ge n(\pi-2\eps)\right) =1-o(1).\]
 So the required result follows from attractiveness of the threshold contact process.
 \end{proof}

\addtocontents{toc}{\protect\setcounter{tocdepth}{0}}
\section*{Acknowledgments}
The author thanks Shankar Bhamidi for various helpful comments and
discussions while writing this article.

\addtocontents{toc}{\protect\setcounter{tocdepth}{1}}
\bibliographystyle{amsalpha}

\begin{bibdiv}
\begin{biblist}

\bib{RA05}{article}{
      author={Albert, R.},
       title={Scale-free networks in cell biology},
        date={2005},
     journal={Journal of cell science},
      volume={118},
      number={21},
       pages={4947\ndash 4957},
}

\bib{AO03}{article}{
      author={Albert, R.},
      author={Othmer, H.G.},
       title={The topology of the regulatory interactions predicts the
  expression pattern of the segment polarity genes in drosophila melanogaster},
        date={2003},
     journal={Journal of Theoretical Biology},
      volume={223},
      number={1},
       pages={1\ndash 18},
}

\bib{AC03}{article}{
      author={Aldana, M.},
      author={Cluzel, P.},
       title={A natural class of robust networks},
        date={2003},
     journal={Proceedings of the National Academy of Sciences},
      volume={100},
      number={15},
       pages={8710},
}

\bib{A94}{article}{
      author={Athreya, K.B.},
       title={Large deviation rates for branching processes--i. single type
  case},
        date={1994},
     journal={The Annals of Applied Probability},
      volume={4},
      number={3},
       pages={779\ndash 790},
}

\bib{B62}{article}{
      author={Bennett, G.},
       title={Probability inequalities for the sum of independent random
  variables},
        date={1962},
     journal={Journal of the American Statistical Association},
       pages={33\ndash 45},
}

\bib{CD11}{article}{
      author={Chatterjee, S.},
      author={Durrett, R.},
       title={Persistence of activity in threshold contact processes, an
  "annealed approximation" of random boolean networks},
        date={2011},
     journal={Random Structures \& Algorithms},
}

\bib{DP86}{article}{
      author={Derrida, B.},
      author={Pomeau, Y.},
       title={Random networks of automata: a simple annealed approximation},
        date={1986},
     journal={EPL (Europhysics Letters)},
      volume={1},
       pages={45},
}

\bib{drossel2005number}{article}{
      author={Drossel, B.},
       title={Number of attractors in random boolean networks},
        date={2005},
     journal={Physical Review E},
      volume={72},
      number={1},
       pages={016110},
}

\bib{drossel20083}{article}{
      author={Drossel, B.},
       title={3 random boolean networks},
        date={2008},
     journal={Reviews of nonlinear dynamics and complexity},
      volume={1},
       pages={69},
}

\bib{D07}{book}{
      author={Durrett, R.},
       title={Random graph dynamics},
   publisher={Cambridge Univ Pr},
        date={2007},
      volume={20},
}

\bib{FK88}{article}{
      author={Flyvbjerg, H.},
      author={Kjr, NJ},
       title={Exact solution of kauffman's model with connectivity one},
        date={1988},
     journal={Journal of Physics A: Mathematical and General},
      volume={21},
       pages={1695},
}

\bib{FH01}{article}{
      author={Fox, J.J.},
      author={Hill, C.C.},
       title={From topology to dynamics in biochemical networks},
        date={2001},
     journal={Chaos: An Interdisciplinary Journal of Nonlinear Science},
      volume={11},
      number={4},
       pages={809\ndash 815},
}

\bib{G79}{book}{
      author={Griffeath, David},
      author={Griffeath, David},
       title={Additive and cancellative interacting particle systems},
   publisher={Springer-Verlag Berlin},
        date={1979},
}

\bib{huang1999gene}{article}{
      author={Huang, S.},
       title={Gene expression profiling, genetic networks, and cellular states:
  an integrating concept for tumorigenesis and drug discovery},
        date={1999},
     journal={Journal of Molecular Medicine},
      volume={77},
      number={6},
       pages={469\ndash 480},
}

\bib{KCA02}{article}{
      author={Kadanoff, L.},
      author={Coppersmith, S.},
      author={Aldana, M.},
       title={Boolean dynamics with random couplings},
        date={2002},
     journal={Arxiv preprint nlin/0204062},
}

\bib{K69}{article}{
      author={Kauffman, S.A.},
       title={Metabolic stability and epigenesis in randomly constructed
  genetic nets},
        date={1969},
     journal={Journal of theoretical biology},
      volume={22},
      number={3},
       pages={437\ndash 467},
}

\bib{K93}{book}{
      author={Kauffman, Stuart~A.},
       title={The origins of order: Self organization and selection in
  evolution},
   publisher={Oxford University Press, USA},
        date={1993},
}

\bib{LR08}{article}{
      author={Lee, D.S.},
      author={Rieger, H.},
       title={Broad edge of chaos in strongly heterogeneous boolean networks},
        date={2008},
     journal={Journal of Physics A: Mathematical and Theoretical},
      volume={41},
       pages={415001},
}

\bib{LS97}{article}{
      author={Luque, B.},
      author={Sol{\'e}, R.V.},
       title={Phase transitions in random networks: simple analytic
  determination of critical points},
        date={1997},
     journal={Physical Review E},
      volume={55},
      number={1},
       pages={257\ndash 260},
}

\bib{MV11}{article}{
      author={Mountford, T.},
      author={Valesin, D.},
       title={Supercriticality for annealed approximations of boolean
  networks},
        date={2010},
     journal={Arxiv preprint arXiv:1007.0862},
}

\bib{NSW01}{article}{
      author={Newman, M.E.J.},
      author={Strogatz, S.H.},
      author={Watts, D.J.},
       title={Random graphs with arbitrary degree distributions and their
  applications},
        date={2001},
     journal={Phys. Rev. E.},
      volume={64},
       pages={026118},
}

\bib{NSW02}{article}{
      author={Newman, M.E.J.},
      author={Watts, D.J.},
      author={Strogatz, S.H.},
       title={Random graph models of social networks},
        date={2002},
     journal={Proceedings of the National Academy of Sciences of the United
  States of America},
      volume={99},
      number={Suppl 1},
       pages={2566\ndash 2572},
}

\bib{POGL09}{article}{
      author={Pomerance, A.},
      author={Ott, E.},
      author={Girvan, M.},
      author={Losert, W.},
       title={The effect of network topology on the stability of discrete state
  models of genetic control},
        date={2009},
     journal={Proceedings of the National Academy of Sciences},
      volume={106},
      number={20},
       pages={8209},
}

\bib{pomerance2009effect}{article}{
      author={Pomerance, A.},
      author={Ott, E.},
      author={Girvan, M.},
      author={Losert, W.},
       title={The effect of network topology on the stability of discrete state
  models of genetic control},
        date={2009},
     journal={Proceedings of the National Academy of Sciences},
      volume={106},
      number={20},
       pages={8209\ndash 8214},
}

\bib{shmulevich2002binary}{article}{
      author={Shmulevich, I.},
      author={Zhang, W.},
       title={Binary analysis and optimization-based normalization of gene
  expression data},
        date={2002},
     journal={Bioinformatics},
      volume={18},
      number={4},
       pages={555\ndash 565},
}

\bib{SL94}{article}{
      author={Sol{\'e}, R.V.},
      author={Luque, B.},
       title={Phase transitions and antichaos in generalized kauffman
  networks},
        date={1994},
     journal={Physics Letters A},
      volume={196},
      number={1-2},
       pages={331\ndash 334},
}

\bib{somogyi1996modeling}{article}{
      author={Somogyi, R.},
      author={Sniegoski, C.},
       title={Modeling the complexity of genetic networks: understanding
  multigenic and pleiotropic regulation},
        date={1996},
     journal={Complexity},
      volume={1},
       pages={45\ndash 63},
}

\bib{S74}{article}{
      author={Stigler, Stephen~M},
       title={Linear functions of order statistics with smooth weight
  functions},
        date={1974},
     journal={The Annals of Statistics},
       pages={676\ndash 693},
}

\bib{tabus2006normalized}{article}{
      author={Tabus, I.},
      author={Jorma, R.},
      author={Astola, J.},
       title={Normalized maximum likelihood models for boolean regression with
  application to prediction and classification in genomics},
        date={2006},
     journal={Computational and Statistical Approaches to Genomics},
       pages={235\ndash 258},
}

\bib{W77}{article}{
      author={Wellner, Jon~A},
       title={A glivenko-cantelli theorem and strong laws of large numbers for
  functions of order statistics},
        date={1977},
     journal={The Annals of Statistics},
       pages={473\ndash 480},
}

\end{biblist}
\end{bibdiv}


\addtocontents{toc}{\protect\setcounter{tocdepth}{0}}

\end{document}